\documentclass[journal,10pt,]{IEEEtran}
\ifCLASSINFOpdf
\else
\fi
\usepackage{amsmath}
\usepackage{amssymb}
\usepackage{algorithm}
\usepackage{algpseudocode}
\usepackage{amsfonts}
\usepackage{graphicx}
\usepackage{amsthm}

\makeatletter
\renewcommand*\env@matrix[1][*\c@MaxMatrixCols c]{%
  \hskip -\arraycolsep
  \let\@ifnextchar\new@ifnextchar
  \array{#1}}
\makeatother  
\newtheorem{definition}{Definition}
\newtheorem{theorem}{Theorem}
\newtheorem{lemma}{Lemma}

\begin{document}
%
\title{ADMM for Distributed Dynamic Beam-forming}
%
%
%

\author{Marie~Maros,~\IEEEmembership{Student Member,~IEEE,}
        and~Joakim~Jald\'{e}n,~\IEEEmembership{Senior Member,~IEEE}
\thanks{M. Maros and J. Jald\'{e}n are with the Department of Signal Processing, KTH Royal Institute of Technology, Stockholm, Sweden}
}

%
%

\markboth{Submitted to IEEE Transactions on Signal and Information Processing over Networks}%
{Shell \MakeLowercase{\textit{et al.}}: Bare Demo of IEEEtran.cls for Journals}
%



\maketitle

\begin{abstract}
This paper shows the capability the alternating direction method of multipliers (ADMM) has to track, in a distributed manner, the optimal down-link beam-forming solution in a multiple input multiple output (MISO) multi-cell network given a dynamic channel. Each time the channel changes, ADMM is allowed to perform one algorithm iteration. In order to implement the proposed scheme, the base stations are not required to exchange channel state information (CSI), but will require to exchange interference values once. We show ADMM's tracking ability in terms of the algorithm's Lyapunov function given that the primal and dual solutions to the convex optimization problem at hand can be understood as a continuous mapping from the problem's parameters. We show that this holds true even considering that the problem looses strong convexity when it is made distributed. We then show that these requirements hold for the down-link, and consequently up-link, beam-forming case. Numerical examples corroborating the theoretical findings are also provided.
\end{abstract}

\begin{IEEEkeywords}
Alternating direction method of multipliers (ADMM), dencentralized optimization, dynamic optimization, MIMO, multi-cell wireless networks, second-order cone programming (SOCP)
\end{IEEEkeywords}

\IEEEpeerreviewmaketitle

%
\IEEEpeerreviewmaketitle

%
%

\section{Introduction}

Coordinated transmissions in multi-cell communication networks has in the recent years drawn great attention due to the promise of significantly higher spectral efficiencies \cite{Rt2006, Lodhi2006}. Such techniques include both multi-point cooperative techniques where mobile users are simultaneously served by several base stations \cite{Lodhi2006}, and inter-cell interference mitigating techniques where base stations coordinate to limit interference to neighboring cells \cite{Dahrouj2010}.

Coordinated transmissions place larger requirements on the availability of accurate channel state information (CSI) throughout the network, and these requirements are often the major hurdle for adoption of coordinated transmission techniques. Centralized solutions further require channel knowledge of the entire network to be present at a single node that will then be capable of obtaining the optimal transmit strategy and distribute it to the base stations that will be using the respective beam-formers. Centralized solutions are impractical for all but very small networks, and, as mentioned in \cite{Joshi2012}, the channels might have changed before the central solution has reached the base stations.

This has led many researchers to consider distributed optimization techniques that circumvent the need for network wide collection of CSI \cite{Dahrouj2010, Joshi2012, Pennanen2011}. Still, distributed optimization techniques are iterative in nature, and their convergence rate and need for interchanging intermediate information over back-haul channels must always be compared to the total amount of back-haul transmissions needed by a centralized solution when assessing their relative merits. If the convergence to an optimal solution is slow or requires an excessive amount of intermediate signaling, a centralized solution may still be preferable, at least within a localized cluster of neighboring cells. This said, one clear advantage of a distributed solution is that, once it has converged, it may be able to continuously adapt to small changes in the CSI with limited intermediate signaling. This is typically very hard to achieve with centralized solutions, as the CSI needs to be redistributed in the network on a time-scale dictated by the channels coherence time.

Motivated by the above, we will, in this paper, study the ability of the popular alternating direction method of multipliers (ADMM) algorithm to dynamically track an optimal down-link beam-forming solution in a multiple input multiple output (MISO) multi-cell network with time-varying channels. We will assume that the base stations are equipped with multiple antennas and that the mobile terminals (users) are equipped with single antennas. The base stations may use channel state information (CSI) to adapt the multi-antenna transmission in order to intelligently mitigate the effect of inter-cell interference.
Given the described scenario, several notions of an optimal transmit strategy have been adopted in the literature. The main differences lie in what one wishes to optimize. In opportunistic formulations, the focus is on finding a transmit strategy that maximizes a utility function of the transmission rate given a fixed power budget. However, such formulations may lead to variable rates which might not be desired for services where a specific quality of service (QoS) needs to be guaranteed. Additionally, utility rate maximization problems have been shown to be NP-hard in general \cite{Liu2011}, which makes characterization of distributed solutions significantly harder. On the contrary, the problem of minimizing the transmit power subject to QoS constraints in terms of the required signal to noise and interference ratios (SINRs) at each mobile terminal, initially believed to be non-convex, was shown to yield optimal solutions through the use of semi-definite relaxation (SDR) \cite{Bengtsson1999}, and was shown to be equivalent to a second-order cone program (SOCP) \cite{Bengtsson1999}, \cite{Wiesel2006}. We will therefore, in this work, consider the QoS constrained beam-forming problem formulation, partially for reasons of tractability.

Algorithms that can solve convex optimization problems in a distributed manner have attracted great interest in the recent years. A tutorial on general decomposition techniques can be found in \cite{Palomar2006}. Primal and dual decomposition are well known classes of techniques to decompose an optimization problem \cite{Boyd2007}. Both classes of decompositions rely on having a master problem and slave sub-problems. The sub-problems are then independently solved in the separate nodes while the master problem is solved iteratively using parameters obtained from the individual sub-problems. Primal and dual decomposition have also been previously applied to the QoS constrained power minimization problem considered herein. Examples include \cite{5425336} where dual decomposition was used, \cite{Pennanen2011} where primal decomposition was used, and also \cite{Shen2012, Joshi2012} where ADMM was used. There are also problem specific distributed techniques based on fixed-point iterations that exploit up-link down-link duality \cite{Dahrouj2010}. The up-link down-link approach has also been extended to the rate maximization problem \cite{Huang2011}. The work in \cite{Shen2012} considered a robust ADMM formulation where SDR was used to solve local worst-case robust beam-forming problems. We will however herein, for simplicity, not consider the robust ADMM formulation and instead apply ADMM as in \cite{Joshi2012}.

ADMM has previously been shown capable of tracking the optimal solution to a dynamically changing optimization problem \cite{Ling2013}.  Such results exist also for other decomposition techniques \cite{Simonetto}. However, most of the available results require strong convexity of the objective function, and deal with either unconstrained minimization problems or static feasible sets \cite{Simonetto, Ling2013}. A notable exception is the work in \cite{Simonetto2014}, where a time varying constraint set is used and the requirement for a strongly convex objective is removed in a gradient-type tracking algorithm. While the centralized QoS beam-forming problem considered herein has a strongly convex objective function, this strong convexity is unfortunately broken in the ADMM decomposition. This makes us unable to directly apply the results in \cite{Ling2013}, as these require strong convexity of the objective function in order to establish linear convergence \cite{Hong2012} as part of the proof therein. Furthermore, the QoS constraints are herein channel dependent and thus time-varying. We will therefore seek to establish a dynamic tracking result through an application of the weaker but more general ADMM convergence results presented in \cite{Boyd2010}.

Another issue to take into consideration is that the QoS constrained beam-forming problem is not generally guaranteed to be feasible over all possible channels for a given user to base station assignations. Clearly, no algorithm will be able to track the optimal solution if it does not exist. We will deal with this issue by limiting the tracking argument to sequences of channels within a compact set of channels for which the problem is guaranteed to be feasible. In practice, a communications system would continuously need to monitor the amount of power used, reject and admit users to the system, and reassign users to base stations. When the channel changes sufficiently much the mechanism in charge of performing the user to base station assignation will naturally introduce a change leading to an abrupt change of the problem structure and implying a loss of the tracking ability. We will however not explicitly consider such mechanisms further, and only consider tracking for channel sequences where the centralized problem remain feasible.

Finally, ADMM as proposed in \cite{Shen2012, Joshi2012}, and many other distributed algorithms as well, will only provide feasible solutions in the limit. This issue has not been overlooked by the research community. The standard solution is to interrupt the algorithm and perform a projection over the feasible set in order to to achieve feasibility of the solution \cite{Joshi2012}, \cite{5425336}. However, even when the original problem is assumed to be feasible, there is no guarantee that the projection step is successful. While it can be argued that the likelihood of the projection being feasible increases as the algorithm converges , we propose an alternative way of addressing this issue by allowing the QoS constraints to be violated by some small amount. As, under stable running condition, the deviation from the QoS constraints will be limited and controlled, we argue that the introduction of a simple QoS SINR margin would be enough to ensure the applicability of the algorithm in practice, and therefore we will not strictly enforce the QoS constraints.

With the above caveats in mind, we will prove that an ADMM algorithm that is allowed to perform one ADMM iteration per discrete unit time will be able to yield beam-formers that are arbitrarily close to the globally optimal beam-formers and provide SINRs which are arbitrarily close to or above  the target QoS constraints, provided that the channels vary sufficiently little between each time step within a compact set of channels for which the overall beam-forming problem is feasible.
  
We begin the paperin Section \ref{section:problemformulation} by introducing the down-link beam-forming problem and its reformulation so as to write it in a way that is amendable to ADMM and in order to introduce notation. We also discuss the requirements and assumptions needed for our main result to hold true in the same section. We then proceed to show in Section \ref{section:trackingwithadmm} the tracking ability of ADMM in a general setting under certain continuity assumption of intermediate solutions when viewed as functions of the channels and intermediate iterates. Once the tracking ability has been shown, we proceed to prove in Section~\ref{section:continuityanalysis} that the continuity assumptions hold for the considered beam-forming problem. Numerical results that are used to illustrate the results are presented in Section \ref{section:numericalresults}. Finally, concluding remarks are given in Section \ref{section:conclusions}.

\section{Problem formulation \label{section:problemformulation}}

Consider a cellular system with $B$ base stations and $K$ users where each user is served by one base station at a time. Assume that each base station is equipped with  $N_T$ transmitting antennas and that each mobile station is equipped with a single antenna. 

Each user $k$ has been assigned to a specific base station $b=b(k)$ that will serve it while keeping the interference caused to other users small. Given channels $\mathbf{h}_{mk} \in \mathbb{C}^{N_T \times 1}$ from base station $m$ to user $k$, the received signal at user $k$ can be expressed as \cite{5425336,Pennanen2011,Joshi2012}
\begin{subequations}
\begin{align}
   y_{k} \triangleq \underbrace{\mathbf{h}^H_{b(k)k}\mathbf{w}_{k}d_k}_{\text{scaled signal of interest}}   
 +\overbrace{\sum_{i \in \mathcal{U}(b(k)) \setminus k} \mathbf{h}^H_{b(k)k}\mathbf{w}_{i}}^{
 \text{intracell interference}} 
   \\ +\underbrace{\sum_{m \neq b(k)} \sum_{i \in \mathcal{U}(m)} \mathbf{h}^H_{mk}\mathbf{w}_{i}}
  _{\text{intercell interference}}+ n_{k} \, ,
\end{align}
\end{subequations}
where $\mathcal{U}(m)$ denotes the set of users served by base station $m$, where $\mathbf{w}_{k}$ denotes the transmit beam-former used by base station $b(k)$ to transmit to user $k$, where $d_k \in \mathbb{C}$ is the signal of interest with $\mathbb{E}\{|d_k|^2\} = 1$ and $\mathbb{E}\{d_kd_j\} = 0$ for $k \neq j$, and where $n_k$ represents circularly symmetric additive white Gaussian noise (AWGN) with variance $\sigma_k^2$. We will assume in the distributed solutions that base station $b$ has knowledge of $\mathbf{h}_{bk}$ for all $k=1,\ldots,K$ and $\mathbf{w}_{k}$ for $k \in \mathcal{U}
(b)$, but not of $\mathbf{h}_{mk}$ for $m \neq b$ and $\mathbf{w}_{k}$ for $k \notin \mathcal{U}
(b)$. The signal to interference and noise ratio (SINR) at user $k$, for a given set of channels and for a given transmit strategy, is given by
\begin{subequations}
\label{eq:defsinr}
 \begin{align}
  & \text{SINR}_k(\mathbf{W}, \mathbf{H}) \triangleq \\ & \frac{|\mathbf{h}_{b(k)k}^H\mathbf{w}_{k}|^2}
 {\displaystyle{\sum_{i \in \mathcal{U}(b(k)) \setminus k}}|\mathbf{h}^H_{b(k)k}\mathbf{w}_{i}|^2 + 
 \sum_{m \neq b} \sum_{i \in \mathcal{U}(m)} |\mathbf{h}_{mk}^H\mathbf{w}_{i}|^2 + \sigma_k^2} \, ,
 \end{align}
\end{subequations}
where $\mathbf{H}$ and $\mathbf{W}$ are matrices containing the complete set of channels and beam-forming vectors.
Since the rate $r_k$ to user $k$ is a monotonically increasing function of $\text{SINR}_k(\mathbf{W}, \mathbf{H})$, requiring a minimum SINR is equivalent to requiring a minimum rate per user. Hence, 
\begin{subequations}
\label{eq:allcentralized}
\begin{align}
\mathcal{W}^{\star}(\mathbf{H}) \triangleq \,
  \underset{\{\mathbf{w}_{k}\}}{\text{arg min}} \qquad   \sum_{b=1}^B\sum_{k \in \mathcal{U}(b)} ||\mathbf{w}_{k}||^2 \label{eq:centralizedobjective}\\
   \text{s.t.} \qquad  \text{SINR}_k(\mathbf{W},\mathbf{H})\geq \gamma_{k} > 0  \label{eq:centralizedSINR}\\
   k = 1,\hdots,K \nonumber
\end{align}
\end{subequations}
can be seen as a formulation of the minimum power strategy for a specific set of user QoS constraints, where $\mathcal{W}^{\star}(\mathbf{H})$ provides the optimal set of beam-formers $\mathbf{W}^{\star}$ given the channels $\mathbf{H}$.

The problem in \eqref{eq:allcentralized} can be equivalently formulated as a second order cone program (SOCP) \cite{Bengtsson1999,Wiesel2006}. The extension to a variable amount of antennas per base station is straightforward and avoided herein for simplicity. However, the extension to several antennas in reception is probably NP-hard for resource allocation with fixed QoS requirements \cite{Bjornson2013} and rate maximization subject to power constraints is NP-hard for two or more transmit antennas per base station \cite{Liu2011}. A technical issue with \eqref{eq:allcentralized} as stated is that the optimal solution is only unique  up to a phase ambiguity in the beam-forming vectors, i.e., if $\mathbf{w}_{k}^{\star}$ is optimal, so is $\mathbf{w}_{k}^{\star}e^{j\phi}$, where $\phi$ is an arbitrary phase. This phase ambiguity is removed when formulating the problems as a SOCP by setting the phase such that the products $\mathbf{h}^{H}_{b(k)k}\mathbf{w}_{k}$ are real valued and positive  \cite{Wiesel2006}, enforcing a unique phase for each beam-former. For this reason we will without loss of generality and without much further comments treat $\mathcal{W}^{\star}(\mathbf{H})$ as a singleton set, i.e., we assume that the optimal solution is unique.

As formulated in \eqref{eq:allcentralized}, the optimization problem would require centralization of the CSI. In order to solve \eqref{eq:allcentralized} with only local CSI,  \cite{Joshi2012, Shen2012}\footnote{The scenario treated in \cite{Shen2012} considers also that the obtained CSI is imperfect which is a generalization we do not consider} proposed ADMM based distributed formulations of problem \eqref{eq:allcentralized}.
Using ADMM in order to solve a problem in a distributed fashion involves creating copies of the variables that are shared by different nodes, or in this case base stations. Hence, the first step is to identify the shared information and define a new set of variables so as to limit the information exchange. We define, similar to  \cite{Shen2012} and \cite{5425336}, 
$t_{mk}^2 \triangleq \sum_{j \in \mathcal{U}(m)}| 
\mathbf{h}_{mk}^H\mathbf{w}_{j}|^2$, for $m \neq b(k)$, which is the power of the inter-cell interference caused by base station $m$ on user $k$ served by base station $b(k)$. The problem in \eqref{eq:allcentralized} can then be equivalently expressed as
\begin{subequations}
\label{eq:alldisready}
\begin{align}
 & \underset{\{t_{mk}\},\{\mathbf{w}_{k}\}}{\text{minimize}} 
 \,\,\,\,\, \qquad \sum_{b = 1}^{B} \sum_{k \in \mathcal{U}(b)}||\mathbf{w}_{k}||^2 \label{eq:costfun} \\
 & \text{s.t.} \,\,\,\,
\frac{|\mathbf{h}^H_{b(k)k}\mathbf{w}_{k}|^2}{\displaystyle{\sum_{i \in \mathcal{U}(b(k)) \setminus k}}
  |\mathbf{h}_{b(k)k}^H\mathbf{w}_{i}|^2 + \sum_{m \neq b} t^{(b(k))2}_{mk} + \sigma_{k}^2}  \geq \gamma_{k} \label{eq:sinrtargets} \\
& \qquad \qquad \qquad \qquad \big(t^{(m)}_{mk}\big)^2 -\sum_{i \in \mathcal{U}(m)}|\mathbf{h}_{mk}^H\mathbf{w}_{i}|^2   \geq  0 \\
 & \qquad \qquad \qquad \qquad \qquad \qquad \qquad  t_{mk}^{(m)} = t_{mk}^{(b(k))} \label{eq:consistency}\\
 & \qquad \qquad \qquad \qquad \qquad \qquad \qquad  \text{for } k = 1,\hdots, K \nonumber \\
 & \qquad \qquad \qquad \qquad \qquad \qquad \qquad  \text{for } m \neq b(k) \nonumber \, ,
\end{align}
\end{subequations}
where $t_{mk}^{(m)}$ is the inter-cell interference copy in base station $m$ and $t_{mk}^{(b(k))}$ is the inter-cell interference copy found in base station $b(k)$. Note that, except for the equality constraints in \eqref{eq:consistency}, i.e., that base stations $m$ and $b(k)$ agree on the amount of interference caused and suffered, the constraints in \eqref{eq:alldisready} only involve information of a single base station and the cost function in \eqref{eq:costfun} is separable across base stations. It should also be clear from the formulation in \eqref{eq:alldisready} that the interference caused by base station $m$ and suffered by a user in base station $b(k)$ will only be relevant, and hence exchanged, among base stations $m$ and $b(k)$. The coupling between base stations is also made explicit by \eqref{eq:consistency}.

Typically, dual decomposition or ADMM are used to decouple problems coupled through a constraint \cite{Palomar2006}. However, in this case we are in the presence of coupling variables. To be able to use ADMM we introduce a consistency variable $\tau_{mk}$ and force the equalities, according to $t_{mk}^{(m)} = \tau_{mk}$ and $t_{mk}^{(b(k))} = \tau_{mk}$.
More compactly, we can define $\mathbf{t}_{b} \in \mathbb{R}^{K + |\mathcal{U}(b)|(B-2)}$ containing base station $b$'s copies of the interference terms caused and suffered by its users, i.e., $t_{bj}^{(b)}$ for $j \not \in \mathcal{U}(b)$ and $t_{mk}$, $\forall m \neq b$ and $\forall k \in \mathcal{U}(b)$, respectively. For notational simplicity we additionally introduce $\mathbf{t}^T = (\mathbf{t}_1^T,\hdots,\mathbf{t}_B^T) \in \mathbf{R}^{2(B-1)K}$ and $\boldsymbol{\tau} \in \mathbb{R}^{K(B-1)}$, as aggregate vectors that contains all interferences and consistency variables. Then, the equality constraints in \eqref{eq:consistency} can be compactly expressed using the equality $\mathbf{E}\boldsymbol{\tau} = \mathbf{t}$, where $\mathbf{E} \in \mathbb{R}^{(B-1)K \times 2(B-1)K}$ is a matrix whose elements are $\{0,1\}$ that copies the elements of $\boldsymbol{\tau}$ in the positions corresponding to the copies in $\mathbf{t}$. If the equality $\mathbf{E}\boldsymbol{\tau} = \boldsymbol{t}$, or equivalently \eqref{eq:consistency}, were to be ignored, \eqref{eq:alldisready} would become decomposable over the base stations since the feasible set would be the Cartesian product of the independent feasible sets. This allows us to use ADMM \cite{Shen2012} (or alternatively dual decomposition \cite{5425336})
to provide a distributed algorithm. In order to simplify the formulation of the problems solved by each of the base stations we introduce

\begin{equation}
\label{eq:distributedSINR}
\begin{aligned}
& \text{SINR}_{k}(\mathbf{W}_b,\mathbf{H}_b,\mathbf{t}_b) \triangleq \\ &\frac{|\mathbf{h}_{b(k)k}^H\mathbf{w}_k|^2}{\sum_{i \in \mathcal{U}(b(k)) \setminus k} |\mathbf{h}_{b(k)k}^H\mathbf{w}_i|^2 + \sum_{m \neq b}t_{mk}^{(b)2} + \sigma_{k}^2 }
\end{aligned}
\end{equation}
and
\begin{equation}
\label{eq:distributedInterference}
\begin{aligned}
\text{INT}_{bj}(\mathbf{W}_b, \mathbf{H}_b, \mathbf{t}_b) \triangleq 
 \big(t_{bj}^{(b)}\big)^2 - \sum_{i \in \mathcal{U}(b)}|\mathbf{h}_{bj}^H\mathbf{w}_i|^2 \, ,
\end{aligned}
\end{equation}
where \eqref{eq:distributedSINR} denotes the SINR of user $k$ as a function of the beam-formers used by base station $b$, $\mathbf{W}_b$, the channels known to base station $b$, $\mathbf{H}_b$ and the estimated caused and suffered interference at base station $b$, $\mathbf{t}_b$. Analogously, \eqref{eq:distributedInterference} represents the constraint on the interference caused by base station $b$ to user $j$, where $j \not \in \mathcal{U}(b)$.

Using these quantities, the problem in \eqref{eq:alldisready} can now be compactly written as
\begin{subequations}
\label{eq:finaldistribute}
\begin{align}
 &\underset{\{\mathbf{w}_k\}, \, \mathbf{t}, \, \boldsymbol{\tau}}{\text{min}} \qquad \qquad \sum_{b=1}^B \sum_{k \in \mathcal{U}(b)}||\mathbf{w}_k||^2 \label{eq:finaldistribute-objective} \\
 &\text{s.t.} \qquad \qquad \text{SINR}_k(\mathbf{W}_b,\mathbf{H}_b,\mathbf{t}_b) \geq \gamma_k \label{eq:sinrtargets2} \\
  & \qquad \qquad \qquad \forall k\in \mathcal{U}(b),\, b=1,\hdots,B \nonumber\\
 & \qquad \qquad \quad \, \text{INT}_{bj}(\mathbf{W}_b,\mathbf{H}_b,\mathbf{t}_b) \geq 0,\, \\   & \qquad \qquad \qquad \forall j \not\in \mathcal{U}(b),\, b=1,\hdots,B  \nonumber \\
 & \qquad \qquad \quad \, \mathbf{E}\boldsymbol{\tau} = \mathbf{t}. \label{eq:finalconsistency}
\end{align}
\end{subequations}
Further, $\mathbf{E}$ can be partitioned accordingly to the $\mathbf{t}_b$ in $\mathbf{t}$ leading to $B$ linear equalities of the kind $\mathbf{E}_b\boldsymbol{\tau} = \mathbf{t}_b$, where $\mathbf{E}_b$ denotes the partition of $\mathbf{E}$ corresponding to the interference terms relevant to base station $b$.

Given a static set of channels, the problem in \eqref{eq:finaldistribute}, or equivalently \eqref{eq:alldisready} or \eqref{eq:allcentralized}, can thus be solved iteratively by Algorithm~\ref{alg:ADMMdistr}, which represents the ADMM algorithm applied to \eqref{eq:finaldistribute}. Convergence to an optimal solution follows from standard convergence proofs such as those presented in \cite{Boyd2010}.

\begin{figure}
\begin{algorithm}[H]
 \caption{ADMM for distributed beamforming}\label{alg:ADMMdistr}
 \begin{algorithmic}[1]
  \State Initialize $ \boldsymbol{\tau}^{[0]}$ and $\boldsymbol{\nu}^{[0]}$ such that $\mathbf{E}^T\boldsymbol{\nu}^{[0]} = 0$. Set $i=1$.
  \State Distributedly solve \label{step:xiterate} 
  \small
  \begin{subequations}
  \label{eq:ADMMdistr}
   \begin{align}
     \underset{\mathbf{t}_b,\{\mathbf{w}_{j}\}}{\text{min}} \,\,\,\, \sum_{j \in \mathcal{U}(b)} 
	||\mathbf{w}_{j}||^2 + (\boldsymbol{\nu}_b^{[i-1]})^T(\mathbf{t}_b - \mathbf{E}_b\boldsymbol{\tau}^{[i-1]}) \label{eq:objective}\\
 \,\,\,\,\,\,\,	+ \frac{\rho}{2}||\mathbf{t}_b - \mathbf{E}_b\boldsymbol{\tau}^{[i-1]}||^2 \notag \\
 \text{s.t.} \, \text{SINR}_k(\mathbf{W}_b,\mathbf{H}_b,\mathbf{t}_b) \geq \gamma_k ,\,
  \forall k \in \mathcal{U}(b) \label{eq:type1}\\
  \qquad \text{INT}_{bj}(\mathbf{W}_b,\mathbf{H}_b,\mathbf{t}_b) \geq 0,\, \forall j \not \in \mathcal{U}(b) \label{eq:type2}
   \end{align}
\end{subequations}
  \normalsize
   for each $b = 1,\ldots,B$ to obtain $\mathbf{w}_{k}^{[i]}$ for $k=1,\ldots,K$ and $\mathbf{t}_b^{[i]}$ at each base station $b$.
  \State Each base stations shares the relevant elements within $\mathbf{t}_b^{[i]}$ with other base stations. \label{step:distribtion}
  \State Compute $\boldsymbol{\tau}^{[i]} = \mathbf{E}^{\dagger}\mathbf{t}^{[i]}$ : $\mathbf{E}_b\boldsymbol{t}_b^{[i]}$ can
  be computed locally by averaging the terms in $\mathbf{t}_b^{[i]}$ with the received terms. \label{step:ziterate} 
  \State Compute $\boldsymbol{\nu}_b^{[i]} = \boldsymbol{\nu}_b^{[i-1]} + \rho(\mathbf{t}_b^{[i]} - \mathbf{E}_b\boldsymbol{\tau}^{[i]})$ \label{step:dualupdate} 
  \State Set $ i \leftarrow i + 1$, and return to 2.
 \end{algorithmic}
\end{algorithm}
\end{figure}


 However, given dynamically fading channels, the risk of rendering the CSI obsolete will lead to the necessity of interrupting the algorithm before it has reached an optimal point \cite{Joshi2012}. In case this happens, the approach proposed in \cite{Joshi2012} and \cite{5425336} is to interrupt the algorithm and to project over the feasible set, by setting the variables $\mathbf{t}^{[i]}_b = \mathbf{E}_b\boldsymbol{\tau}^{[i-1]},\,\,\forall b$. However, this projection is not necessarily feasible in which case more iterations will be required \cite{Shen2012}. The approach advocated herein is instead to allow for the SINR constraints in \eqref{eq:sinrtargets2} or \eqref{eq:sinrtargets} to be violated by a controlled amount.
 
Assuming block fading, and that the changes in the channels from block to block are bounded, a possible solution is to track the optimal set of beam-formers. A result showing ADMM's tracking capabilities, when the objective function changes from iteration to iteration, is provided in \cite{Ling2013}. However, the analysis found in \cite{Ling2013} considers unconstrained minimization of a strongly convex function with a Lipschitz continuous gradient. Unfortunately, even though the original problem in \eqref{eq:allcentralized} can be written with a strongly convex objective function with respect to the beam-formers, the price to pay for decomposability is the loss of strong convexity in \eqref{eq:finaldistribute-objective} with respect to the variables $\mathbf{t}$, $\boldsymbol{\tau}$. The results provided in \cite{Ling2013} heavily rely on ADMM's linear convergence \cite{Deng2015} which has the same requirements. Hence, in order to prove that ADMM is capable of tracking the optimal set of beam-formers, a different approach is required.

Our aim in this paper is to prove that given an initial set of variables $\boldsymbol{\tau}^{[0]}$ and $\boldsymbol{\nu}^{[0]}$ in Algorithm \ref{alg:ADMMdistr} satisfying\footnote{Which is also fulfilled by the optimal set of multipliers $\boldsymbol{\nu}^{\star}$} $\mathbf{E}^T\boldsymbol{\nu}^{[0]} = \boldsymbol{0}$,  ADMM is with only one ADMM iteration per channel change capable of providing a set of beam-formers that lie in a bounded neighborhood of the optimal set of beam-formers while the feasibility SINR constraints in \eqref{eq:sinrtargets2} are violated at most by a bounded amount. The proposal is hence to simply use Algorithm~\ref{alg:ADMMdistr} with the static channels $\mathbf{H}$ replaced by the channels at iteration $i$, denoted by $\mathbf{H}^{[i]}$. In order to prove the tracking capability, we require that the channels lies within a \emph{compact} set of channels, $\mathcal{H}$, that ensures that \eqref{eq:allcentralized} is strictly feasible. The compact set of channels fulfilling this condition will be referred in the sequel as the $\gamma_k-$feasible channels. 

An essential difference compared with other works \cite{Joshi2012, 5425336} is the requirement of strictly feasible channels. Considering strictly feasible channels guarantees that an arbitrarily small change in the channel will not render the problem infeasible. A second difference is that we allow the SINR constraints to be violated by a bounded amount so as to allow for small disagreements in the interference values at different base stations and hence avoiding the need to solve non-feasible problems. 

The contributions of this paper are particularized for the MISO optimal beam-forming problem. However the proof found in Section \ref{section:trackingwithadmm} shows that ADMM is capable of tracking an optimal solution as long as some continuity conditions are met by the problem at hand. In particular, we require that the optimal primal and dual points are continuous functions of the problem's data, which in this case is the channel. Additionally, we also require that the primal parameters, in this case $\mathbf{W}^{[i]}, \mathbf{t}^{[i]}$ and $\boldsymbol{\tau}^{[i]}$ obtained by solving \eqref{eq:ADMMdistr} in step~\ref{step:xiterate} of Algorithm~\ref{alg:ADMMdistr}, are continuous functions of the channel and the previous parameters used by ADMM i.e. $\boldsymbol{\nu}^{[i-1]}$ and $\boldsymbol{\tau}^{[i-1]}$.
In order to formalize the paper's main result we introduce Theorem \ref{maintheorem} which is proven in the subsequent sections.
\begin{theorem} \label{maintheorem} Let $\{ \mathbf{H}^{[i]} \}_{i=0}^\infty$ be a sequence of channels that lie within a compact set $\mathcal{H}$ of strictly $\gamma_k$-feasible channels. Given arbitrary positive constants $\epsilon_1 > 0$ and $\epsilon_2 > 0$, there is some $\delta > 0$ for which Algorithm~1 generates a sequence of beam-formers $\mathbf{W}^{[i]}$ for which the distance to the $\gamma_{k}$-optimal  beam-formers is guaranteed to fulfill
\begin{equation}
\underset{i \rightarrow \infty}{\text{lim sup}}\,||\mathbf{W}^{[i]} - \mathbf{W}^{[i]\star}||_\mathrm{F}^2 \leq \epsilon_1,
\end{equation}
where $\mathbf{W}^{[i]\star}$ denotes the optimal beam-formers at time $i$,
and the SINR of all users are guaranteed to fulfill
\begin{equation}
\underset{i \rightarrow \infty}{\text{lim inf}}\,\,\, \mathrm{SINR}_k(\mathbf{W}^{[i]},\mathbf{H}^{[i]},\mathbf{t}_b^{[i]})  \geq \gamma_k  - \epsilon_2
\end{equation}
whenever $\|\mathbf{H}^{[i]} - \mathbf{H}^{[i-1]} \| \leq \delta$ for all $i \geq 1$.
\end{theorem}

\section{Tracking with ADMM \label{section:trackingwithadmm}}

In this section we show, given a set of continuity assumptions, that ADMM is capable of tracking. In order to be able to show tracking without resorting to any proof requiring linear convergence, we use the convergence proof found in \cite{Boyd2010}. This proof relies on a per-iteration decrease on the algorithm's Lyapunov function, which we define as
\begin{equation}
\label{eq:lyapunov}
V(\boldsymbol{\nu},\boldsymbol{\tau},\mathbf{H}) \triangleq \frac{1}{\rho}||\boldsymbol{\nu} - \boldsymbol{\nu}^{\star}(\mathbf{H})||^2 + \rho||\mathbf{E}(\boldsymbol{\tau} - \boldsymbol{\tau}^{\star}(\mathbf{H}) )||^2,
\end{equation}
where $\boldsymbol{\tau}^{\star}(\mathbf{H})$ denotes the optimal consistency variables in \eqref{eq:finalconsistency} given channels $\mathbf{H}$, and $\boldsymbol{\nu}^{\star}(\mathbf{H})$ denotes the optimal dual variables associated with the consistency constraints \eqref{eq:finalconsistency} given the channels $\mathbf{H}$. Note that the dependence on $\mathbf{H}$ has its origin in the fact that the optimal interference values and optimal dual multipliers associated with \eqref{eq:finalconsistency} depend on the problems data, i.e. $\mathbf{H}$.
 In \cite{Boyd2010}, ADMM is shown to converge by using the fact that
\begin{equation}
\label{eq:staticlyapunov}
\begin{aligned}
V(\boldsymbol{\nu}^{[i]},\boldsymbol{\tau}^{[i]},\mathbf{H}) \leq V(\boldsymbol{\nu}^{[i-1]},\boldsymbol{\tau}^{[i-1]},\mathbf{H}) \\ 
- \rho||\mathbf{r}^{[i]}|| - \rho ||\mathbf{E}(\boldsymbol{\tau}^{[i]}-\boldsymbol{\tau}^{[i-1]})||^2,
\end{aligned}
\end{equation}
where $\mathbf{r}^{[i]} \triangleq \mathbf{t}^{[i]} - \mathbf{E}\boldsymbol{\tau}^{[i-1]}$ is the primal residual at the $i^{\text{th}}$ iterate. The residual $\mathbf{r}^{[i]}$ represents the disagreement among base stations, or deviation from the mean interference value, during the previous iterate. Equation \eqref{eq:staticlyapunov} essentially implies that there exists a non-zero decrease in $V(\boldsymbol{\nu},\boldsymbol{\tau},\mathbf{H})$ at each iteration unless all base-stations agree on the amount of interference. Note that in \cite{Boyd2010} the optimization problem is assumed to be static, i.e., in our context, the channels $\mathbf{H}$ in \eqref{eq:staticlyapunov} are assumed to be constant from iteration to iteration.  However, in the remainder of this section this assumption will be relaxed and the channel will be assumed to change from one iteration to the next and will hence be indexed using the iteration number. 

Given that the considered set of $\gamma_k-$feasible channels $\mathcal{H}$ is compact and that the channel variation is such that $||\mathbf{H}^{[i]} - \mathbf{H}^{[i+1]}||\leq \delta$ for all $i \geq 0$, we will assume the following:
\begin{itemize}
\item[(A1)] The optimal consistency variables $\boldsymbol{\tau}^{\star}$ for the consistency constraint \eqref{eq:finalconsistency} are a continuous function of the channel $\mathbf{H}$, i.e. $\boldsymbol{\tau}^{\star}(\mathbf{H})$ is a continuous function of $\mathbf{H}$ over $\mathcal{H}$.
\item[(A2)] The optimal dual multipliers $\boldsymbol{\nu}^{\star}$ in the consistency constraint \eqref{eq:finalconsistency} are a continuous function of the channel $\mathbf{H}$, i.e. $\boldsymbol{\nu}^{\star}(\mathbf{H})$ is a continuous function of $\mathbf{H}$ over $\mathcal{H}$.
\item[(A3)] The primal iterates $\mathbf{W}^{[i]}$, $\mathbf{t}^{[i]}$ and $\boldsymbol{\tau}^{[i]}$ at time $i$, are continuous functions of the iterates at time $i-1$, i.e., $\mathbf{W}^{[i]}(\boldsymbol{\tau}^{[i-1]},\boldsymbol{\nu}^{[i-1]},\mathbf{H}^{[i]})$, $\mathbf{t}^{[i]}(\boldsymbol{\tau}^{[i-1]},\boldsymbol{\nu}^{[i-1]},\mathbf{H}^{[i]})$ and $\boldsymbol{\tau}^{[i]}(\boldsymbol{\tau}^{[i-1]},\boldsymbol{\nu}^{[i-1]},\mathbf{H}^{[i]})$ corresponding to the resulting parameters generated by steps \ref{step:xiterate} and \ref{step:ziterate} in Algorithm \ref{alg:ADMMdistr}, are continuous functions of their respective input parameters.
\end{itemize} 
Note that assumption (A3) also implies continuity of $\boldsymbol{\nu}^{[i]}(\boldsymbol{\tau}^{[i-1]},\boldsymbol{\nu}^{[i-1]},\mathbf{H}^{[i-1]})$ by the continuity of the dual update in step \ref{step:dualupdate} in Algorithm \ref{alg:ADMMdistr}. Additionally the continuity of $\boldsymbol{\tau}^{[i]}(\boldsymbol{\tau}^{[i-1]},\boldsymbol{\nu}^{[i-1]},\mathbf{H}^{[i]})$ follows by the same principle from the continuity of $\mathbf{t}^{[i]}(\boldsymbol{\tau}^{[i-1]},\boldsymbol{\tau}^{[i-1]},\mathbf{H}^{[i-1]})$ for the beam-forming problem. However, this might not be the case for other optimization problems, and is thus assumed. Assumptions (A1)-(A3) will be proven to hold in the next section. However, for the time being they will be assumed to be given.

Conceptually, the proof that follows can be split in two parts. First, we show that given a bound on the Lyapunov function before the ADMM update, i.e. $V(\boldsymbol{\tau}^{[i-1]},\boldsymbol{\nu}^{[i-1]},\mathbf{H}^{[i]})$, we are capable of guaranteeing a bound on the distance to the optimal set of beam-formers. Second, we then show that there exists a channel variation $\delta$ such that we are guaranteed that the bound on the Lyapunov function holds in the limit when $i \rightarrow \infty$. Following this approach, we introduce two lemmas and their respective proofs to show that Theorem \ref{maintheorem} holds true given assumptions (A1)-(A3).
\begin{lemma} Given that assumption (A3) holds and given a constant $\epsilon_1 > 0$, there exists a constant $c > 0$ such that
\begin{equation} \label{eq:Lyapanov-bound}
\begin{aligned}
& \underset{i \rightarrow \infty}{\text{lim sup}} & V(\boldsymbol{\tau}^{[i-1]},\boldsymbol{\nu}^{[i-1]},\mathbf{H}^{[i]}) \leq c \\
\end{aligned}
\end{equation} 
implies that
\begin{equation}
\begin{aligned}
& \underset{i \rightarrow \infty}{\text{lim sup}} & 
\end{aligned} ||\mathbf{W}^{[i]} - \mathbf{W}^{[i]\star}||_{\mathrm{F}}^2 \leq \epsilon_1\,. 
\end{equation}
\end{lemma}
Two alternative proofs of this can be provided. In the general case, the bound in \eqref{eq:Lyapanov-bound} will by \eqref{eq:lyapunov} imply that $\boldsymbol{\nu}^{[i-1]}$ and $\boldsymbol{\tau}^{[i-1]}$ are close to their respective optimal values. The continuity assumption for $\mathbf{W}^{[i]}(\boldsymbol{\tau}^{[i-1]},\boldsymbol{\nu}^{[i-1]},\mathbf{H}^{[i-1]})$ made in (A3) will imply that also $\mathbf{W}^{[i]}$ is close to the global optimal value. However, for the particular problem at hand, an explicit bound that yields insight into the dependency of $c$ on $\epsilon_1$ can also be provided; which is done in Appendix \ref{section:appendix1}.
\begin{lemma} \label{lemma:boundedlyapunov} Given that assumptions (A1)-(A3) hold, given a compact set $\mathcal{H}$ of $\gamma_k$-feasible channels, and given a constant $c > 0$, there exists a maximum channel variation $\delta > 0$, where $\mathbf{H}^{[i-1]} \in \mathcal{H}$ and $\| \mathbf{H}^{[i]} - \mathbf{H}^{[i-1]} \| \leq \delta$ for all $i \geq 1$, for which
\begin{equation}
\begin{aligned}
&\underset{i \rightarrow \infty}{\text{lim sup}} & V(\boldsymbol{\tau}^{[i-1]},\boldsymbol{\nu}^{[i-1]},\mathbf{H}^{[i]}) \leq c\,.
\end{aligned}
\end{equation}
\end{lemma}
\begin{proof}
ADMM guarantees that $V(\boldsymbol{\nu}^{[i]},\boldsymbol{\tau}^{[i]},\mathbf{H}^{[i]}) < V(\boldsymbol{\nu}^{[i-1]},\boldsymbol{\tau}^{[i-1]},\mathbf{H}^{[i]})$  for the Lyapunov function defined in \eqref{eq:lyapunov}, unless $\boldsymbol{\nu}^{[i-1]}$ and $\boldsymbol{\tau}^{[i-1]}$ are already optimal for $\mathbf{H}^{[i]}$ \cite{Boyd2010}. Additionally, given assumptions (A1)-(A2), the Lyapunov function $V(\boldsymbol{\nu},\boldsymbol{\tau},\mathbf{H})$ is continuous in $\mathbf{H}$ for fixed $\boldsymbol{\nu}$ and $\boldsymbol{\tau}$. The proof of Lemma~\ref{lemma:boundedlyapunov} presented below uses these facts to confine the Lyapunov function between two values. This can be achieved by guaranteeing that an increase in the Lyapunov function due to a change in the channel will always countered by a decrease in the Lyapanov function due to one iteration of the ADMM algorithm. A bound on the maximum variation of $V(\boldsymbol{\nu},\boldsymbol{\tau},\mathbf{H})$ over $\mathbf{H} \in \mathcal{H}$ for which $\|  \mathbf{H} - \mathbf{H}^{[i]}\| \leq \delta$, and for any pair $(\boldsymbol{\nu},\boldsymbol{\tau})$ that could be generated by the algorithm,
is obtained together with a minimum guaranteed decrease provided by ADMM. This is possible due to continuity assumptions (A1)-(A3) and the compactness of $\mathcal{H}$.

To this end, assume that the channel variation from iteration to iteration is upper bounded by some $\delta > 0$ to be specified later, i.e. $||\mathbf{H}^{[i+1]} - \mathbf{H}^{[i]}|| \leq \delta$ for all $i \geq 0$. For some given $\mu_{\text{l}} > 0$ and arbitrary $(\boldsymbol{\nu}^{[0]},\boldsymbol{\tau}^{[0]})$ with $\mathbf{E}^T\boldsymbol{\nu}^{[0]} = \mathbf{0}$ define
\begin{equation}
\mu_0 \triangleq \text{max}\left\{V(\boldsymbol{\nu}^{[0]},\boldsymbol{\tau}^{[0]},\mathbf{H}^{[0]}),\mu_{\text{l}}\right\}.
\end{equation}
Then, choose a finite $\Delta\mu > 0,$ let $\mu_{\text{u}} \triangleq \mu_0 + \Delta \mu $ and define the set $\mathcal{V} \subset \mathbb{R}^{3(B-1)K}$ as
\begin{equation}
\mathcal{V} \triangleq \{(\boldsymbol{\nu},\boldsymbol{\tau})\, |\, \exists \mathbf{H} \in \mathcal{H}, V(\boldsymbol{\nu},\boldsymbol{\tau},\mathbf{H}) \leq \mu_{\text{u}}, \mathbf{E}^T\boldsymbol{\nu} = \mathbf{0}\}.
\end{equation}
Due to the compactness of $\mathcal{H}$, the continuity of $\boldsymbol{\nu}^{\star}(\mathbf{H})$ and $\boldsymbol{\tau}^{\star}(\mathbf{H})$, and the strong convexity of $V(\boldsymbol{\nu},\boldsymbol{\tau},\mathbf{H})$ in $(\boldsymbol{\nu},\boldsymbol{\tau})$, the set $\mathcal{V}$ is also compact. Next, let
\begin{equation}
\mathcal{U} \triangleq \{(\boldsymbol{\nu},\boldsymbol{\tau},\mathbf{H}) \,|\, V(\boldsymbol{\nu},\boldsymbol{\tau},\mathbf{H}) \geq \mu_{\text{l}},\, \mathbf{E}^T\boldsymbol{\nu} = \mathbf{0}\}.
\end{equation}
Note that this set is closed but not bounded. However, the set $\mathcal{U} \cap (\mathcal{V} \times \mathcal{H})$ is compact as it is closed and bounded. The set  $\mathcal{U} \cap (\mathcal{V} \times \mathcal{H})$ is simply the set of parameters $(\boldsymbol{\nu},\boldsymbol{\tau},\mathbf{H})$ for which the Lyapunov function \eqref{eq:lyapunov} is upper and lower bounded according to
\begin{equation}
\label{eq:lyapunovconfinement}
\mu_{\text{l}} \leq V(\boldsymbol{\nu},\boldsymbol{\tau},\mathbf{H}) \leq  \mu_{\text{u}} = \mu_0 + \Delta \mu  \, , 
\end{equation}
i.e. we are confining the Lyapunov function's value between $\mu_{\text{l}}$ and $\mu_{\text{u}}$ by considering $(\boldsymbol{\nu},\boldsymbol{\tau},\mathbf{H}) \in \mathcal{U} \cap (\mathcal{V} \times \mathcal{H})$.

 From \cite{Boyd2010} we have that for a given triplet $(\boldsymbol{\nu}^{[i-1]},\boldsymbol{\tau}^{[i-1]},\mathbf{H}^{[i]})$ at the start of steps \ref{step:xiterate} in Algorithm \ref{alg:ADMMdistr}, the decrease in the Lyapanov function in the $i^\text{th}$ iteration over steps \ref{step:xiterate} to \ref{step:ziterate}, is lower bounded as [cf.\ \eqref{eq:staticlyapunov}]
 \begin{equation} \label{eq:decreaselowerbound}
 \begin{aligned}
 D(\boldsymbol{\tau}^{[i-1]},\boldsymbol{\nu}^{[i-1]}, \mathbf{H}^{[i]}) & \triangleq \\
 V(\boldsymbol{\nu}^{[i-1]},\boldsymbol{\tau}^{[i-1]},\mathbf{H}^{[i]}) - V(\boldsymbol{\nu}^{[i]},\boldsymbol{\tau}^{[i]},\mathbf{H}^{[i]}) & \geq \\
   \rho||\mathbf{r}^{[i]}||^2 + \rho||\mathbf{E}(\boldsymbol{\tau}^{[i]} - \boldsymbol{\tau}^{[i-1]})||^2 & \geq 0
 \end{aligned}
\end{equation} 
where equality holds only at the optimum when $\boldsymbol{\tau}^{[i-1]} = \boldsymbol{\tau}^\star(\mathbf{H}^{[i]})$, $\boldsymbol{\nu}^{[i-1]} = \boldsymbol{\nu}^\star(\mathbf{H}^{[i]})$  and $V(\boldsymbol{\nu}^{[i-1]},\boldsymbol{\tau}^{[i-1]},\mathbf{H}^{[i]}) = 0$. 

 
 
Given (A3) the lower bound on $D(\boldsymbol{\tau}^{[i-1]},\boldsymbol{\nu}^{[i-1]}, \mathbf{H}^{[i]})$ in \eqref{eq:decreaselowerbound} is continuous in $(\boldsymbol{\tau}^{[i-1]},\boldsymbol{\nu}^{[i-1]},\mathbf{H}^{[i]})$. Additionally, given the compactness of $\mathcal{U} \cap (\mathcal{V} \times \mathcal{H})$ and that $V(\boldsymbol{\nu}^{[i-1]},\boldsymbol{\tau}^{[i-1]},\mathbf{H}^{[i]}) \geq \mu_\text{l} > 0$ whenever $(\boldsymbol{\tau}^{[i-1]},\boldsymbol{\nu}^{[i-1]},\mathbf{H}^{[i]}) \in \mathcal{U} \cap (\mathcal{V} \times \mathcal{H})$,
it follows that there is a constant $d > 0$ for which
$
D(\boldsymbol{\nu},\boldsymbol{\tau},\mathbf{H}) \geq d
$
for all $(\boldsymbol{\tau}^{[i-1]},\boldsymbol{\nu}^{[i-1]},\mathbf{H}^{[i]}) \in \mathcal{U} \cap (\mathcal{V} \times \mathcal{H})$,
i.e., there is a minimum guaranteed decrease of the Lyapanov function.

By assumption it holds that $(\boldsymbol{\nu}^{[0]},\boldsymbol{\tau}^{[0]},\mathbf{H}^{[0]}) \in \mathcal{V} \times \mathcal{H}$ and that $V(\boldsymbol{\tau}^{[0]},\boldsymbol{\tau}^{[0]},\mathbf{H}^{[0]}) \leq \mu_0$. Assume now that for some arbitrary $i \geq 1$ it holds that $(\boldsymbol{\nu}^{[i-1]},\boldsymbol{\tau}^{[i-1]},\mathbf{H}^{[i-1]}) \in \mathcal{V} \times \mathcal{H}$ and that $V(\boldsymbol{\nu}^{[i-1]},\boldsymbol{\tau}^{[i-1]},\mathbf{H}^{[i-1]}) \leq \mu_0$. Given a change in the channel from $\mathbf{H}^{[i-1]}$ to $\mathbf{H}^{[i]}$, we can have that $V(\boldsymbol{\nu}^{[i-1]},\boldsymbol{\tau}^{[i-1]},\mathbf{H}^{[i]}) \leq \mu_{\text{l}}$ or that $V(\boldsymbol{\nu}^{[i-1]},\boldsymbol{\tau}^{[i-1]},\mathbf{H}^{[i]}) > \mu_{\text{l}}$. In the former case, it follows immediately by the monotonicity of the Lyapanov function for fixed channels that also $V(\boldsymbol{\nu}^{[i]},\boldsymbol{\tau}^{[i]},\mathbf{H}^{[i]}) \leq \mu_{\text{l}}$. In the latter case,
the Lyapunov function can be bounded, by using the fact that $\mathbf{E}^T\boldsymbol{\nu}^{\star}= \mathbf{0}$ and the triangular inequality applied to \eqref{eq:lyapunov}, according to 
\begin{equation}
\label{eq:boundincrlyap}
\begin{aligned}
V(\boldsymbol{\nu}^{[i-1]},\boldsymbol{\tau}^{[i-1]},\mathbf{H}^{[i]}) \leq   \mu_0 + 2\sqrt{\mu_0}(\frac{1}{\sqrt{\rho}}\Delta\nu^{\star}(\delta) + \sqrt{\rho} \Delta \tau^{\star}(\delta))   \\
   + \frac{1}{\rho}(\Delta\nu^{\star}(\delta))^2 + \rho (\Delta \tau^{\star}(\delta))^2,
\end{aligned}
\end{equation}
where
\begin{subequations}
\begin{align}
\Delta\tau^{\star}(\delta) \triangleq & \underset{\mathbf{H},\mathbf{H}' \in \mathcal{H}}{\text{max}} & ||\boldsymbol{\tau}^{\star}(\mathbf{H}) - \boldsymbol{\tau}^{\star}(\mathbf{H}')|| \\
\ & \text{s.t.} & ||\mathbf{H}-\mathbf{H}'|| \leq \delta
\end{align}
\end{subequations}
and where $\Delta\nu^{\star}(\delta)$ is analogously defined. Due to the compactness of $\mathcal{H}$, the continuity of $\boldsymbol{\tau}^{\star}(\mathbf{H})$ and $\boldsymbol{\nu}^{\star}(\mathbf{H})$ the quantities $\Delta \tau^{\star}(\delta)$ and $\Delta \nu^{\star}(\delta)$ are bounded and satisfy $\underset{\delta \rightarrow 0}{\text{lim}} \quad \Delta \tau ^{\star}(\delta) = 0$ and $\underset{\delta \rightarrow 0}{\text{lim}} \quad \Delta \nu^{\star}(\delta)=0.$


We need to select $\delta$ so as to guarantee that $(\boldsymbol{\nu}^{[i]},\boldsymbol{\tau}^{[i]}, \mathbf{H}^{[i-1]}) \in \mathcal{U} \cap (\mathcal{V} \times \mathcal{H})$. We do this by selecting $\delta$ such that 

\begin{equation}
\label{eq:deltacondition1}
\begin{aligned}
2\sqrt{\mu_0}(\frac{1}{\sqrt{\rho}}\Delta \nu^{\star}(\delta) + \sqrt{\rho}\Delta \tau^{\star}(\delta))  \\ +  \frac{1}{\rho}(\Delta\nu^{\star}(\delta))^2 + \rho (\Delta \tau^{\star}(\delta))^2 \leq \Delta \mu,
\end{aligned}
\end{equation}
implying that we have that  $V(\boldsymbol{\nu}^{[i]},\boldsymbol{\tau}^{[i]},\mathbf{H}^{[i]}) \leq V(\boldsymbol{\nu}^{[i-1]},\boldsymbol{\tau}^{[i-1]},\mathbf{H}^{[i]}) - d$. Thus, if $\delta$ is chosen such that
\begin{equation}
\label{eq:deltacondition2}
\begin{aligned}
2\sqrt{\mu_0}(\frac{1}{\sqrt{\rho}}\Delta \nu^{\star}(\delta) + \sqrt{\rho}\Delta \tau^{\star}(\delta))  \\ +  \frac{1}{\rho}(\Delta\nu^{\star}(\delta))^2 + \rho (\Delta \tau^{\star}(\delta))^2 \leq d,
\end{aligned}
\end{equation}
we have that $V(\boldsymbol{\nu}^{[i]},\boldsymbol{\tau}^{[i]},\mathbf{H}^{[i]}) \leq \mu_0$, which implies in turn that $(\boldsymbol{\nu}^{[i+1]},\boldsymbol{\tau}^{[i+1]},\mathbf{H}^{[i+1]}) \in \mathcal{V} \times\mathcal{H}$. Therefore, $\delta$ can be selected small enough so as to guarantee that the decrease can always compensate for the increase induced by the channel change, and at the same time, guarantee that there exists no channel change that pushes the Lyapunov function to a region in which the guaranteed decrease does not apply.  Note that $d$ is not dependent on $\delta$ but on $\mathcal{H}$, while $\mu_{\text{l}}$ and $\Delta \mu$ are arbitrarily selected. Hence it is always possible to find a parameter $\delta > 0$ fulfilling \eqref{eq:deltacondition1} and \eqref{eq:deltacondition2}.

Expressions \eqref{eq:deltacondition1} and \eqref{eq:deltacondition2} provide insights on how to select the parameter $\rho$ in case one can obtain the sensitivity of the dual problem or primal problem with respect to the problem's data; in other words, if the dual problem were to be very sensitive to the problem's data while the primal is less, one would select a large value for $\rho$.

It follows by induction that the bound $V(\boldsymbol{\nu}^{[i]},\boldsymbol{\tau}^{[i]},\mathbf{H}^{[i]}) \leq \mu_0$ will hold for all $i$. Additionally, if $\delta'$ is picked so that we have a margin, i.e such that $d \geq 2\sqrt{\mu_0}(\frac{1}{\sqrt{\rho}}\Delta \nu^{\star}(\delta') + \sqrt{\rho} \Delta \tau^{\star}(\delta')) + \frac{1}{\rho}(\Delta \nu^{\star }(\delta'))^2 + \rho (\Delta \tau^{\star}(\delta'))^2 + m,$ 
where $m > 0$ is a constant, we have that at each iteration, as long as we remain within $\mathcal{U} \cap (\mathcal{V} \times \mathcal{H})$, a net decrease, i.e. 
\begin{equation}
V(\boldsymbol{\nu}^{[i]},\boldsymbol{\tau}^{[i]},\mathbf{H}^{[i]}) \geq V(\boldsymbol{\nu}^{[i-1]},\boldsymbol{\tau}^{[i-1]},\mathbf{H}^{[i-1]}) - m.
\end{equation}
This implies that there exists a $i_0$ such that:
\begin{equation}
V(\boldsymbol{\nu}^{[i_0]},\boldsymbol{\tau}^{[i_0]},\mathbf{H}^{[i_0]}) \leq \mu_{\text{l}}.
\end{equation}
Note that we are not guaranteed a decrease of at least $d$ now since $(\boldsymbol{\nu}^{[i_0]},\boldsymbol{\tau}^{[i_0]},\mathbf{H}^{[i_0]}) \not\in \mathcal{U}$. However, we have that after one channel iteration, the Lyapunov function will be upper bounded by:
\begin{subequations}
\begin{align}
V(\boldsymbol{\nu}^{[i_0]},\boldsymbol{\tau}^{[i_0]},\mathbf{H}^{[i_0 + 1]})  \leq \mu_{\text{l}} + 2\sqrt{\mu_{\text{l}}}(\frac{1}{\sqrt{\rho}} \Delta \nu^{\star}(\delta') \\ + \sqrt{\rho}\Delta \tau^{\star}(\delta'))    + \frac{1}{\rho}(\Delta \nu^{\star} (\delta'))^2 + \rho(\Delta \tau^{\star}(\delta'))^2.
\end{align}
\end{subequations}
In case $V(\boldsymbol{\nu}^{[i_0]},\boldsymbol{\tau}^{[i_0]},\mathbf{H}^{[i_0 + 1]}) \geq \mu_{\text{l}}$ we have that $(\boldsymbol{\nu}^{[i_0]},\boldsymbol{\tau}^{[i_0]},\mathbf{H}^{[i_0+1]}) \in \mathcal{U} \cap (\mathcal{V} \times \mathcal{H})$ and therefore, there is a guaranteed decrease $d > 2\sqrt{\mu_0}(\frac{1}{\sqrt{\rho}}\Delta \nu^{\star}(\delta') + \sqrt{\rho} \Delta \tau^{\star}(\delta')) + \frac{1}{\rho}(\Delta \nu^{\star }(\delta'))^2 + \rho (\Delta \tau^{\star}(\delta'))^2$ compensating the increase caused in the Lyapunov function and yielding that $\forall\, i \geq i_0$
\begin{equation}
V(\boldsymbol{\nu}^{[i]},\boldsymbol{\tau}^{[i]},\mathbf{H}^{[i]}) \leq \mu_l,
\end{equation}
or equivalently
\begin{equation}
\begin{aligned}
\label{eq:limitlyapunov}
& \underset{i \rightarrow \infty}{\text{lim sup}} & V(\boldsymbol{\nu}^{[i]},\boldsymbol{\tau}^{[i]},\mathbf{H}^{[i]}) \leq \mu_{\text{l}};
\end{aligned}
\end{equation}
and thus
\begin{equation}
\label{eq:finalboundlyapunov}
\begin{aligned}
\underset{i \rightarrow \infty}{\text{lim sup}}  \qquad V(\boldsymbol{\nu}^{[i]},\boldsymbol{\tau}^{[i]},\mathbf{H}^{[i+1]}) \leq  c
\end{aligned}
\end{equation}
where  $c \triangleq \mu_{\text{l}} + \sqrt{2 \mu_l} (\frac{1}{\sqrt{\rho}}\Delta \nu^{\star}(\delta')  +\sqrt{\rho}\Delta \tau^{\star}(\delta'))  + \frac{1}{\rho}(\Delta \nu^{\star}(\delta'))^2 + \rho(\Delta \tau^{\star}(\delta'))^2$ ,
 concluding the proof of Lemma \ref{lemma:boundedlyapunov}. 
\end{proof}

Given ADMM's nature, primal feasibility can not be guaranteed until the algorithm has converged completely for a fixed channel. We therefore proceed to show the worse case possible deviation from the desired SINRs, $\gamma_k$.

In particular, in the worst case scenario the obtained SINR for user $k$ satisfies for $i \geq i_0$
\begin{equation}
\label{eq:SINRbound}
\text{SINR}_k(\mathbf{W}_b,\mathbf{H}_b,\mathbf{t}_b) =\gamma_k \left( 1 -\left( \frac{ 4 c}{\rho \sigma_k^2 + 4c} \right) \right) = \gamma_k \left(1 - \epsilon_{2k} \right).
\end{equation}
The proof of \eqref{eq:SINRbound} can be found in appendix \ref{section:appendix2}. This concludes the fact that ADMM can track the optimal set of beam-formers given the continuity assumptions (A1)-(A3) and that the channel does not vary too much from one time instance to the next.
Note that when considering the minimum achieved SINR due to the disagreement among base stations, the noise's variance plays an important role, i.e. the larger the noise variance the more negligible the disagreement among base stations is. As one might intuitively expect, the parameter $\rho$ is relevant in order to mitigate the disagreement. This can be easily seen due to the penalty parameter assigning weight to the term $||\mathbf{E}_b\mathbf{t}_b - \boldsymbol{\tau}||^2$ in \eqref{eq:objective}. However, from \eqref{eq:SINRbound} we can see that the effect $\rho$ has is equivalent to that of a ``noise enhancer" when it comes to mitigating the effect of the disagreement on the interference values. 
\section{Continuity analysis \label{section:continuityanalysis}}
In this section we show that assumptions (A1)-(A3) made in order to prove ADMM's tracking ability of the optimal set of beam-formers hold. We will first argue that showing continuity of the optimal consistency variables $\boldsymbol{\tau}^{\star}$ (A1) and the optimal dual variables associated with \eqref{eq:finalconsistency}, $\boldsymbol{\tau}^{\star}$ (A2) is equivalent to showing continuity of the primal and dual optimal variables of the centralized problem. When it comes to the optimal consistency variables $\boldsymbol{\tau}^{\star}$ this is fairly obvious, since the interference constraints \eqref{eq:type2}, as defined in Algorithm \ref{alg:ADMMdistr} will be fulfilled tightly at the optimal point. In order to show the continuity of the optimal dual multipliers $\boldsymbol{\nu}^{\star}$ associated with \eqref{eq:finalconsistency} it suffices to express the Lagrangian of the centralized problem \eqref{eq:allcentralized} and of the problem in \eqref{eq:finaldistribute} as in  \cite{Dahrouj2010}. By finding the optimality conditions with respect to the additional variables (i.e., the interference estimates $\mathbf{t}$ and the consistency variables $\boldsymbol{\tau}$) one can show that each of the elements in $\boldsymbol{\nu}^{\star}$ equals the product of the corresponding interference estimate $t_{mbk}^{(m)\star}$ and the multiplier associated with the SINR constraint \eqref{eq:type1} that contains it. Further, by taking gradient with respect to the beam-formers we obtain that the dual multipliers corresponding to the SINR constraint \eqref{eq:type1} equal the dual multipliers associated to the SINR constraints \eqref{eq:centralizedSINR} in the centralized problem. 

We will first show that the optimal interference estimates $t_{mbk}^{\star}$ are continuous functions of the channel. In order to do this we show that the optimal set of beam-formers $\mathbf{W}^{\star}$ are continuous functions of the channels in the centralized problem.
For this purpose we require Theorem \ref{theorem:continuity} (a special case of \cite[Theorem 2.2, 2.3]{Fiacco1990}). For completeness we include the definitions of \emph{closed} and \emph{open} point to set mappings defined as in \cite{Fiacco1990}.
\begin{definition}
\label{def:closed}
A point to set mapping $\mathcal{W}(\mathbf{H})$ is \emph{closed} at $\mathbf{\bar{H}}$ if for any sequence of channels $\mathbf{H}_n \in \mathcal{H}$, $\mathbf{H}_n \rightarrow \bar{\mathbf{H}}$,  and associated feasible beam-formers $\mathbf{W}_n \in \mathcal{W}(\mathbf{H}_n)$ such that $\mathbf{W}_n \rightarrow \bar{\mathbf{W}}$ it holds that $\bar{\mathbf{W}} \in \mathcal{W}(\bar{\mathbf{H}})$.
\end{definition}
\begin{definition}
A point to set mapping $\mathcal{W}(\mathbf{H})$ is \emph{open} at $\mathbf{\bar{H}}$ if for any sequence of channels $\mathbf{H}_n \in \mathcal{H}$,  such that $\mathbf{H}_n \rightarrow \mathbf{\bar{H}}$ and $\bar{\mathbf{W}} \in \mathcal{W}(\bar{\mathbf{H}})$, it holds that there exists $m$ and $\{\mathbf{W}_n\}$ such that $\mathbf{W}_n \in \mathcal{W}(\mathbf{H}_n)$ for all $n \geq m$, and $\mathbf{W}_n \rightarrow \mathbf{\bar{W}}$. 
\end{definition}
We are now ready to introduce the following theorem:
\begin{theorem}
\label{theorem:continuity}
Let $\mathcal{W}(\mathbf{H})$ and $\mathbf{W}^{\star}(\mathbf{H})$ be the set of feasible of beam-formers and the optimal beam-formers given channel $\mathbf{H}$ respectively. For the problem in \eqref{eq:allcentralized} $\mathbf{W}^{\star}(\mathbf{H})$ is continuous at $\mathbf{H}$ if:
\begin{enumerate}
\item The objective function in \eqref{eq:centralizedobjective} is continuous on $\mathcal{W}(\mathbf{H})$;
\item The point to set mapping $\mathcal{W}(\mathbf{H})$ is closed and open at $\mathbf{H}$;
\item The primal optimal point exists and is unique.
\end{enumerate}
\end{theorem}
The objective function in \eqref{eq:allcentralized} is strictly convex and continuous on $\mathbb{C}^{N_T \times K }$. Additionally when writing the constraints as SOCs the phase ambiguity, otherwise present, vanishes as mentioned in Section \ref{section:problemformulation}; this implies that not only the first, but the also the third of the theorem hold. Hence, in order to show continuity of the function $\mathbf{W}^{\star}(\mathbf{H})$ with respect to the channels, we require showing that the point to set mapping $\mathcal{W}(\mathbf{H})$ is closed and open $\forall \mathbf{H} \in \mathcal{H}$.  

Intuitively, showing that the point to set mapping representing the feasible sets is closed implies that an arbitrarily small change in the channel is not capable of violating the constraints by an arbitrarily large amount.
\begin{lemma}
\label{lemma:primalcont}
Given the set of $\gamma_k-$feasible channels $\mathcal{H}$, for a specific user-base station allocation and SINR requirements, the mapping providing the optimal set of beam-formers is closed and open.
\end{lemma} 
\begin{proof}
Note that the SINR constraints of the centralized problem \eqref{eq:centralizedSINR} are continuous functions of the channels and the beam-formers. In particular, let $\text{SINR}_k(\mathbf{W}_n,\mathbf{H}_n)$ in \eqref{eq:defsinr} be the value of the SINR constraint of user $k$ given the channel $\mathbf{H}_n$ and the set of beam-formers $\mathbf{W}_n$ and let $\mathbf{\bar{H}}$ and $\mathbf{\bar{W}}$ be given as in Definition \ref{def:closed}. Then, due to continuity of the SINR functions with $\mathbf{H}$ and $\mathbf{W}$ it follows that
\begin{equation}
\underset{n \rightarrow \infty}{\text{lim}}\, \text{SINR}_k(\mathbf{W}_n,\mathbf{H}_n) = \text{SINR}_k(\bar{\mathbf{W}},\mathbf{\bar{H}}),\, \forall k,
\end{equation}
and given that all pairs $\mathbf{H}_n, \mathbf{W}_n$ satisfy $\text{SINR}_k(\mathbf{W}_n,\mathbf{H}_n) \geq \gamma_k$ this will also be satisfied in the limit. This establishes that the mapping is closed.

We now proceed to show that $\mathcal{W}(\mathbf{H})$ is also open. In order to prove this, we assume that we have a set of feasible beam-formers $\mathbf{\bar{W}}$ for a given channel $\mathbf{\bar{H}}$. Then, we use known results to establish a neighborhood of $\mathbf{\bar{H}}$ for which we can find a scaling vector $\mathbf{p} = [p_1,\hdots,p_K]^T$ that is a continuous function of the channel such that a feasible set of beam-formers $\mathbf{W} \in \mathcal{W}(\mathbf{H})$ can be found using $\mathbf{\bar{W}}$ as $\mathbf{w}_k = \sqrt{p_k}\mathbf{\bar{w}}_k.$ For this purpose, assume that for a given channel $\mathbf{\bar{H}}$ we have available (without loss of generality) strictly feasible beam-formers $\mathbf{\bar{W}} \in \mathcal{W}(\mathbf{\bar{H}})$. Define now as in \cite{Boche2002}, the power of the interference caused by the transmission to user $j$ over user $k$  as $G_{kj}(\mathbf{\bar{H}}, \mathbf{\bar{W}}) =\mathbf{\bar{w}}_{j}^{H}\mathbf{\bar{h}}_{b(j)k}\mathbf{\bar{h}}_{b(j)k}^H\mathbf{\bar{w}}_{j}$. These terms will be collected in the matrix $\boldsymbol{\Psi}(\mathbf{\bar{H}},\mathbf{\bar{W}}),$ where
\begin{equation}
[\boldsymbol{\Psi}(\mathbf{\bar{H}},\mathbf{\bar{W}})]_{kj} = \begin{cases}
G_{kj}(\mathbf{\bar{H}},\mathbf{\bar{W}}), & j \neq k\\
0 & j = k.
\end{cases}
\end{equation}
Let $\mathbf{D}(\mathbf{\bar{H}},\mathbf{\bar{W}}) \triangleq \text{diag}\{\frac{\gamma_1}{G_{11}(\mathbf{\bar{H}},\mathbf{\bar{W}})},\hdots,\frac{\gamma_K}{G_{KK}(\mathbf{\bar{H}},\mathbf{\bar{W}})}\}$ which represents the ratio between desired SINR and received power for each of the users, and let $\boldsymbol{\sigma} = [\sigma_1^2,\hdots,\sigma_K^2]^T$. The vector $\mathbf{p}$ is defined as the optimal solution to the power allocation problem 
\begin{equation}
\begin{aligned}
\label{eq:boche}
& \underset{\mathbf{p}}{\text{min}}\,\,||\mathbf{p}'||_1 & \text{s.t.}\,\, \text{SINR}'_k(\mathbf{p}',\sigma_k^2,\mathbf{\bar{W}},\mathbf{H}) \geq \gamma_k,\, \forall k,
\end{aligned}
\end{equation}
where $\text{SINR}'_k$ denotes the SINR of user $k$ given a set of fixed beam-formers $\mathbf{\bar{W}}$ a set of channels $\mathbf{\bar{H}}$ and noise variances $\sigma_k^2$.  From \cite{Boche2002} we know that the solution to \eqref{eq:boche} (if it exists) to the optimization problem \eqref{eq:boche} is unique and characterized by
\begin{equation}
\label{eq:optscaling}
(\mathbf{I} - \mathbf{D}(\mathbf{H},\mathbf{\bar{W}})\boldsymbol{\Psi}(\mathbf{H},\mathbf{\bar{W}}))\mathbf{p} = \mathbf{D}(\mathbf{H},\mathbf{\bar{W}})\boldsymbol{\sigma},
\end{equation}
which is an equivalent formulation of the SINR constraints \eqref{eq:centralizedSINR} being fulfilled tightly. 
Additionally, given theorem 4 in \cite{Boche2002} we have that the solution $\mathbf{p} > 0$ exists if and only if the matrix $\mathbf{D}(\mathbf{\bar{H}},\mathbf{\bar{W}})\boldsymbol{\Psi}(\mathbf{\bar{H}},\mathbf{\bar{W}})$ has spectral radius strictly smaller than 1, i.e. $\rho(\mathbf{D}(\mathbf{\bar{H}},\mathbf{\bar{W}})\boldsymbol{\Psi}(\mathbf{\bar{H}},\mathbf{\bar{W}})) < 1$.
In other words, by solving \eqref{eq:optscaling} we are capable of finding a feasible set of beam-formers for $\mathbf{H} \neq \mathbf{\bar{H}}$ in a neighborhood of $\mathbf{\bar{H}}$ by scaling the strictly feasible beam-formers for $\mathbf{\bar{H}}$ by the powers obtained using $\mathbf{p} = (\mathbf{I} - \mathbf{D}(\mathbf{H},\mathbf{\bar{W}})\boldsymbol{\Psi}(\mathbf{H},\mathbf{\bar{W}}))^{-1}\mathbf{D}(\mathbf{H},\mathbf{\bar{W}})\boldsymbol{\sigma}$, i.e. $\mathbf{w}_k = \sqrt{p_k^{\star}}\mathbf{\bar{w}}_k$, relying on the fact that the spectral radius, $\rho(\mathbf{D}(\mathbf{\bar{H}},\mathbf{\bar{W}})\boldsymbol{\Psi}(\mathbf{\bar{H}},\mathbf{\bar{W}})),$ is \emph{strictly} smaller than 1. Note that the scaling can be claimed to be continuous with the channels because of the continuity of the matrices $\boldsymbol{\Psi}(\mathbf{H},\mathbf{\bar{W}})$ and $\mathbf{D}(\mathbf{H},\mathbf{\bar{W}})$ with the channels.
This argument establishes that for a neighborhood of $\mathbf{\bar{H}}$ that fulfills 
\begin{equation}
\mathcal{N}(\mathbf{\bar{H}}) \triangleq \{\mathbf{H}| \rho(\mathbf{D}(\mathbf{H},\mathbf{\bar{W}})\boldsymbol{\Psi}(\mathbf{H},\mathbf{\bar{W}})) < 1 \},
\end{equation}
there exists a continuous scaling $\mathbf{p}$ given by \eqref{eq:optscaling} providing feasible beam-formers. This implies that given a sequence of channels $\mathbf{H}_n \rightarrow \mathbf{\bar{H}}$ and $\mathbf{\bar{W}} \in \mathcal{W}(\mathbf{\bar{H}})$, there exists an $m$ and a sequence $\{\mathbf{W}_n\}$ such that for all $n \geq m$ $\mathbf{W}_n \in \mathcal{W}(\mathbf{H}_n)$. In particular, given a feasible beam-former $\mathbf{\bar{W}} \in \mathcal{W}(\mathbf{\bar{H}})$ we can generate by using \eqref{eq:optscaling} feasible beam-formers in the neighborhood of $\mathbf{\bar{H}}$ thus concluding the proof. By invoking Theorem \ref{theorem:continuity} the continuity of the function $\mathbf{W}^{\star}(\mathbf{H})$ follows.
\end{proof}

We now introduce and prove the following lemma regarding the continuity of the dual multipliers.
\begin{lemma}
\label{lemma:dualcont}
Given the set of $\gamma_k-$feasible channels $\mathcal{H}$ for a specific user base station allocation and SINR requirements, the optimal dual multiplers $\{\lambda_k^{\star}\},\,k=1,\hdots,K$ are continuous functions of the channel $\mathbf{H}$.
\end{lemma}
\begin{proof}
As shown in \cite{Dahrouj2010} and reviewed in \cite{Bj2014}  the dual problem of \eqref{eq:allcentralized} can be expressed as
\begin{subequations}
\label{eq:dual}
\begin{align}
& \underset{\{\lambda_{k}\}}{\text{min}} \qquad \qquad \qquad \qquad \qquad  \sum_{b =1}^B\sum_{k \in \mathcal{U}(b)} \lambda_{k}\sigma_{k}^2 \\
& \text{s.t.} \quad \mathbf{I} + \sum_{j=1}^{K} \lambda_{j}\mathbf{h}_{b(k)j}\mathbf{h}_{b(k)j}^H \succeq (1 + \frac{1}{\gamma_{k}})\lambda_k \mathbf{h}_{b(k)k}\mathbf{h}_{b(k)k}^H\,, \label{eq:dualsinr}\\
& \qquad \qquad \qquad \qquad \qquad k =1,\hdots, K.
\end{align}
\end{subequations}
The dual problem in \eqref{eq:dual} has been shown \footnote{Note that \cite{Dahrouj2010} contains a technical error in the proof of Theorem 1 between equations (12) and (13). However, this does not compromise the validity of the result. An alternative proof can be provided by using the fact that given a symmetric positive semidefinite matrix $\mathbf{A} \in \mathbb{C}^{n \times n}$ and an $n \times 1$ vector $\mathbf{b}$ in the row space of $\mathbf{A},$ $\mathbf{A} \succeq \mathbf{bb}^H$ iff $\mathbf{b}^H\mathbf{A}^{\dagger}\mathbf{b} \leq 1.$} to be equivalent to solving the following up-link beam-forming problem yielding the down-link up-link duality result in \cite{Dahrouj2010}:
\begin{subequations}
\label{eq:uplinkdual}
\begin{align}
& \underset{\{\lambda_k\}}{\text{min}} & \sum_{k=1}^K \lambda_k \sigma_k^2 \\
& \text{s.t.} & \eta_k \geq \sigma_k \label{eq:uplinksinr},
\end{align}
\end{subequations}
where $\eta_{k} = \underset{\mathbf{w}_{k}^d: ||\mathbf{w}_k^d|| = 1}{\text{max}}\frac{\lambda_{k}|\mathbf{w}_{k}^{dH}\mathbf{h}_{b(k)k}|^2}{\sum_{ j \neq k} \lambda_{j}|\mathbf{w}_{k}^{dH}\mathbf{h}_{b(k)j}|^2 + \sigma_k^2||\mathbf{w}_k^d||^2}$ and where $\mathbf{w}_k^d$ denote the up-link beam-formers. It can be shown \cite[Theorem 1]{Dahrouj2010}, that given a fixed set of values $\{\lambda_j\}$ the down-link beam-formers maximizing $\eta_k$ follows the expression $\mathbf{w}_{k}^d(\{\lambda_j\}) = \frac{\left( \sum_{j} \lambda_{j}\mathbf{h}_{b(k)j}\mathbf{h}_{b(k)j}^H + \sigma_k^2\mathbf{I}\right)^{-1}\mathbf{h}_{b(k)k}}{||\left( \sum_{j} \lambda_{j}\mathbf{h}_{b(k)j}\mathbf{h}_{b(k)j}^H + \sigma_k^2\mathbf{I}\right)^{-1}\mathbf{h}_{b(k)k}||},$ and that $\mathbf{w}^d_k$ are a scaled version of the optimal down-link beam-formers for the optimal multipliers $\lambda_k^{\star},$ i.e. $\mathbf{w}_{k}^{\star} = \sqrt{p_{k}^{\star}}\mathbf{w}^d_k(\{\lambda_j^{\star}\}),$ where $p_{k}^{\star}$ is the optimal power allocation to user $k$. This establishes, by  Lemma \ref{lemma:primalcont} uniqueness and continuity with the channels of the normalized up-link beam-formers, i.e. the mapping
\begin{equation}
\begin{aligned}
&\mathcal{W}^{d\star}(\mathbf{H}) \triangleq \\ &\left\{ \mathbf{w}_{k}^{d\star} \,|\, \mathbf{w}_{k}^{d\star} = \frac{\left( \sum_{j} \lambda_{j}^{\star}\mathbf{h}_{b(k)j}\mathbf{h}_{b(k)j}^H + \sigma_k^2\mathbf{I}\right)^{-1}\mathbf{h}_{bk}}{||\left( \sum_{j} \lambda_{j}^{\star}\mathbf{h}_{b(k)j}\mathbf{h}_{b(k)j}^H + \sigma_k^2\mathbf{I}\right)^{-1}\mathbf{h}_{b(k)k}||} \right\},
\end{aligned}
\end{equation}
 where $\{\lambda_k^{\star}\}$ are the optimal solutions to \eqref{eq:dual}, is closed and open for all $\mathbf{H} \in \mathcal{H}$. This establishes the continuity of the function $\mathbf{W}^{\star d}(\mathbf{H})$.
By uplink-downlink duality \cite{Dahrouj2010} the optimal dual multipliers $\{\lambda_j^{\star}\}$ correspond to the scaled powers of the up-link beam-formers, i.e. $||\mathbf{w}_k^{\star d}||^2 = \lambda_k \sigma^2_k$, which by the continuity of $\mathbf{w}_k^{\star d}$ establishes the continuity of $\lambda_k^{\star}$ and proves the Lemma.
\end{proof}

The continuity of the optimal and primal dual variables of the centralized problem in \eqref{eq:allcentralized} has now been proven. It is possible to show not only continuity, but differentiability of the function mapping channels to the unique optimal primal-dual solution by writing the problem as a standard SOCP and using the results in \cite{Alizadeh2003}. This yields Lipschitz continuity within a set of compact channels since it would allow bounding the Jacobian matrix's largest eigenvalue. However, this result is more involved and not required for the proofs at hand.

We now proceed to prove that (A3) holds. For this purpose, we require proving continuity of the elements $\mathbf{w}^{[i]},\,t^{[i]}_{mk}$, resulting at each iteration $i$ of ADMM, with respect to the parameters fed to the algorithm  at iteration $i$, i.e. consistency variables $\boldsymbol{\tau}^{[i-1]}$ and duals $\boldsymbol{\nu}^{[i-1]}$ resulting from iterate $i-1$ and with respect to the channels. This will imply  continuity of all the ADMM parameters $\boldsymbol{\tau},\boldsymbol{\nu}$ since the rest of the steps are updated linearly with $t^{[i]}_{mk}$. We proceed, using the same methodology as in the centralized case, to show that the optimization problem in Algorithm \ref{alg:ADMMdistr} yields continuous primal solutions.

\begin{lemma}
\label{lemma:distrcont}
Given the set of $\gamma_k-$feasible channels $\mathcal{H}$ for a specific user-base station allocation and SINR requirements and the parameters $\boldsymbol{\tau}^{[i-1]}$ and $\boldsymbol{\nu}^{[i-1]}$ provided for iterate $i$ in algorithm \ref{alg:ADMMdistr}, the ADMM parameters provided for iterate $i$, i.e. $\boldsymbol{\tau}^{[i]},$  $\boldsymbol{\nu}^{[i]}$ and the obtained primal solution $\mathbf{w}^{[i]},\,\mathbf{t}^{[i]}$ are continuous functions of $\mathbf{H}$, $\boldsymbol{\tau}^{[i-1]}$ and of $\boldsymbol{\nu}^{[i-1]}$.
\end{lemma}
\begin{proof}
 Let us equivalently (in the sense that it yields an equivalent problem) rewrite the objective function in \eqref{eq:objective} as $\sum_{k \in \mathcal{U}(b)} ||\mathbf{w}_k||^2 + \frac{\rho}{2}||\mathbf{t}_b - \mathbf{E}_b\boldsymbol{\tau}^{[i-1]} + \frac{\boldsymbol{\nu}_b}{\rho}||^2$. For simplicity define $\mathbf{y}_b^{[i-1]} \triangleq \mathbf{E}_b\boldsymbol{\tau}^{[i-1]} - \frac{\boldsymbol{\nu}_b^{[i-1]}}{\rho}$. Note now that the interference constraints \eqref{eq:type2}) might not always hold tightly as opposed to the SINR constraints  \eqref{eq:type1}. This is due to the fact that $t_{bj}^{(b)}$, appearing in \eqref{eq:type2}, is selected to fulfill the constraint, but at the same time to be close to the corresponding value in $\mathbf{y}_{b}$ as possible so as to minimize the objective. Hence, given a set of values $(\{\mathbf{w}_k\},\{t_{mk}^{(b)}\}_{m \neq b,k \in  \mathcal{U}(b)})$, corresponding to the beam-formers and suffered interference values, fulfilling the SINR constraints the problem will always be feasible since the caused interference values $\{t_{bj}^{(b)}\}_{j \not \in \mathcal{U}(b)}$ can always be selected accordingly. For this reason, the coming analysis will prioritize the fulfillment of the SINR constraints \eqref{eq:type1} and deal with the interference constraints \eqref{eq:type2} later on.

The conditions required to establish continuity are uniqueness of the primal solution, and that the feasible set
\begin{equation}
\label{eq:primalset}
\mathcal{WT}(\mathbf{H}) \triangleq \{(\mathbf{W},\mathbf{t}) | \text{\eqref{eq:type1} and \eqref{eq:type2} hold } \forall k,j,b \},
\end{equation}
 corresponding to the feasible sets of beam-formers and estimated interference values $\mathbf{t}_b$, is both closed and open for all $\mathbf{H}$ and $\mathbf{y}  \triangleq \mathbf{E}\boldsymbol{\tau} - \frac{\boldsymbol{\nu}}{\rho}$. Note that the feasible set does not explicitly depend on the parameter $\mathbf{y}$ since the feasibility of a beam-former will not be affected by $\mathbf{y}$.

 The proof that $\mathcal{WT}(\mathbf{H})$ is closed is analogous to the centralized case (proof of Lemma \ref{lemma:primalcont}) and will therefore be omitted. Uniqueness of the primal solution follows from the strong convexity of the objective function in \eqref{eq:objective}. The proof of $\mathcal{WT}(\mathbf{H})$ being open is very similar but provides some insight to the problem and will therefore be included.
 Given a set of channels $\mathbf{\bar{H}}$ and parameters $\mathbf{\bar{y}}$, assume each base station has performed an iteration of the ADMM algorithm and found the corresponding optimal solutions. We will then have that all SINR constraints \eqref{eq:type1} will hold tightly, while all interference constraints \eqref{eq:type2} will be either not active, weakly active, or active depending on the values in $\mathbf{\bar{y}}$. Analogously to before, given an optimal point, the SINR constraints, corresponding to all users and therefore all base stations,  can be expressed as $(\mathbf{I} - \mathbf{D}(\mathbf{H},\mathbf{W})\boldsymbol{\Psi}(\mathbf{H},\mathbf{W}))\mathbf{1} = \mathbf{D}(\mathbf{H},\mathbf{W}) \boldsymbol{\eta}$, where in this case $\boldsymbol{\eta} \triangleq  (\sum_{m \neq b}(t_{mj_1}^{(m)[i]})^2 + \sigma_{j_1}^2 ,\hdots, \sum_{m \neq b}(t_{mj_{|\mathcal{U}(b)|}}^{(m)[i]})^2) + \sigma_{j_{|\mathcal{U}(b)|}}^2)^T,$ where $j_k \in \mathcal{U}(b)$.  Given a second set of channels and parameters $\mathbf{H}$ and $\mathbf{y}$, if the optimal set of beam-formers and interference levels corresponding to $\mathbf{\bar{H}}$ and $\mathbf{\bar{y}}$ where used, the SINR constraints \eqref{eq:type1} may again not be fulfilled. We will circumvent this in the same way as before, implying therefore, that there will exist a scaling $\mathbf{p}$ that is continuous with the channel and allows us to produce a feasible set of beam-formers for $\mathbf{H}$ based on the optimal beam-formers for $\mathbf{\bar{H}}$. Note however, that with this new scaling, if the interference values are left untouched, and the interference constraints \eqref{eq:type2} might not be fulfilled. A simple way of solving this problem is to define the new interference values as $t_{bj}^{(b)2} = \sqrt{\text{max}_j(p_j)}t_{bj}^{(b)[i]2}$. From here, the proof is analogous to that of the centralized case.
\end{proof}
In this case, the problem can also be re-written as a standard SOCP and degeneracy conditions can be studied. However, one of the conditions required to show continuity and differentiability of the function that maps the channels to the optimal values is strict complementarity which is not fulfilled in this case when the interference constraints \eqref{eq:type2} are weakly active. In these cases differentiability of the mapping to the primal dual optimal solution is lost but, as proven, continuity is kept.

We have therefore proven, by showing that all point to set mappings representing the feasible sets are open and closed, using the fact that the optimal solutions are unique and invoking Theorem \ref{theorem:continuity} that all assumptions required for the tracking abilities of ADMM when deprived of strong convexity in \eqref{eq:alldisready} hold.
\section{Numerical experiments \label{section:numericalresults}}
This section provides numerical experiments to demonstrate the performance of the proposed dynamic beam-forming technique.  In the simulations the parameter $\gamma_k$ is set to 10 for all users and the plotted SINR (linear) corresponds to the average SINR achieved by the users using the beam-formers obtained after a single iteration. The initial channel vectors (in $\mathbf{H}^{[0]}$) and the innovation channels ($\mathbf{H}^{\text{inn}[i]}$) are generated following: $\mathbf{h}_{bk} \sim \mathcal{CN}(\mathbf{0}, \mathbf{I})$, i.e. they are independent complex circularly Gaussian random vectors with unit variance. A track is then generated as $\mathbf{H}^{[i]} = \sqrt{\zeta}\mathbf{H}^{\text{inn}[i]} + \sqrt{1-\zeta}\mathbf{H}^{[i-1]}$; however, this method might lead to channels that do not have a feasible solution for the required SINRs. In order to avoid tracking infeasible solutions each channel is checked for feasibility prior to feeding it to Algorithm \ref{alg:ADMMdistr}. In case the channel does not allow for a feasible solution it is discarded and replaced by a channel generated following the same innovation equation that is feasible. Even though this is not an appropriate choice when modeling the dynamics of a wireless channel, it allows us to illustrate the tracking ability of the algorithm for this specific problem while keeping the model simple. Note also that this update rule does not guarantee that a bound $||\mathbf{H}^{[i]} - \mathbf{H}^{[i-1]}|| \leq \epsilon$ is fulfilled. This is due to the fact that the innovation $\mathbf{H}^{\text{inn}[i]}$ can take arbitrarily large values. However, it will hold true that $\mathbb{E}\{||\mathbf{H}^{[i]} - \mathbf{H}^{[i-1]}||\} \leq \epsilon .$

The considered system consists of 2 base stations equipped with 4 antennas serving 2 users each.  In all cases, ADMM is initialized with $\boldsymbol{\tau}^{[0]} = \mathbf{0}$ and $\boldsymbol{\nu}^{[0]} = \mathbf{0}$. In case ADMM where to be used with a very large penalty parameter $\rho$ the solution would very slowly deviate from the zero forcing solution since we would be enforcing, initially, that the algorithm does not deviate from it. In figures \ref{fig:singletrack1}, \ref{fig:singletrack50} and \ref{fig:singletrack1000} the dynamic behaviour of the algorithm is illustrated for penalty parameters $\rho =1$, $\rho=50$ and $\rho= 1000$ respectively. It can be seen that the algorithm is in general capable of providing a set of beam-formers which use a similar total power as in the optimal case. We can also see that, even though the solution is not always feasible the achieved SINR levels are not far from 10 ($\gamma_k$). Additionally, as intuitively expected, the fulfillment of the SINR constraints is better as $\rho$ increases. This is due to the fact that we are assigning more weight to consensus among interference levels by selecting a larger $\rho$. Note, however, that by selecting $\rho$ we do not only select how important it is for us that the base stations are in agreement, but also the step size of the sub-gradient step in charge of maximizing the dual problem. Hence, a large value of $\rho$ implies a large step size which might make convergence slow. The same can happen in case of selecting a very small $\rho$. 
In order to illustrate this clearly we provide figures \ref{fig:average1}, \ref{fig:average50} and \ref{fig:average1000} where 10000 independent tracks consisting of 50 channel are generated and averaged, with $\rho=1$, $\rho = 50$ and $\rho=1000$, respectively. 
In this case the tracks are averaged and in case of the SINR we provide a 1 standard deviation shift from the average SINR in all cases. Observe that in figure \ref{fig:average1000} the dynamic solution yields, in average, powers superior to optimal even when the problem is not feasible, while in \ref{fig:average1} and \ref{fig:average50} the dynamic solution approaches the optimal on the opposite side, in other words, it is below the optimal solution. This is due to the selected initial values. A relatively small $\rho$ does not penalize the algorithm from deviation of the initial value, $\mathbf{0}$, and hence ADMM is free to select a power minimizing solution. On the other hand, when the parameter $\rho$ is larger, the solution provided by ADMM will be more similar to a zero-forcing solution, providing a feasible solution earlier but approaching the optimal value from above. 

As mentioned earlier, given a very small or large $\rho$ convergence is slow, however, when the step size is large, the SINR values are very close to feasibility. 
As in the static case, optimal parameter selection for $\rho$ is not known \cite{Boyd2010} except for specific cases \cite{Ghadimi2013}. In \cite{Joshi2012} it is experimentally shown that penalty parameters related to the channels provide quicker convergence. This could also be done in order to improve convergence of the ADMM algorithm, potentially improving the tracking ability. However, in order to normalize $\rho$ with respect to the problem's data we would require to centralize the CSI breaking the distributed nature of the algorithm.

\begin{figure}
\centering\includegraphics[width=21pc]{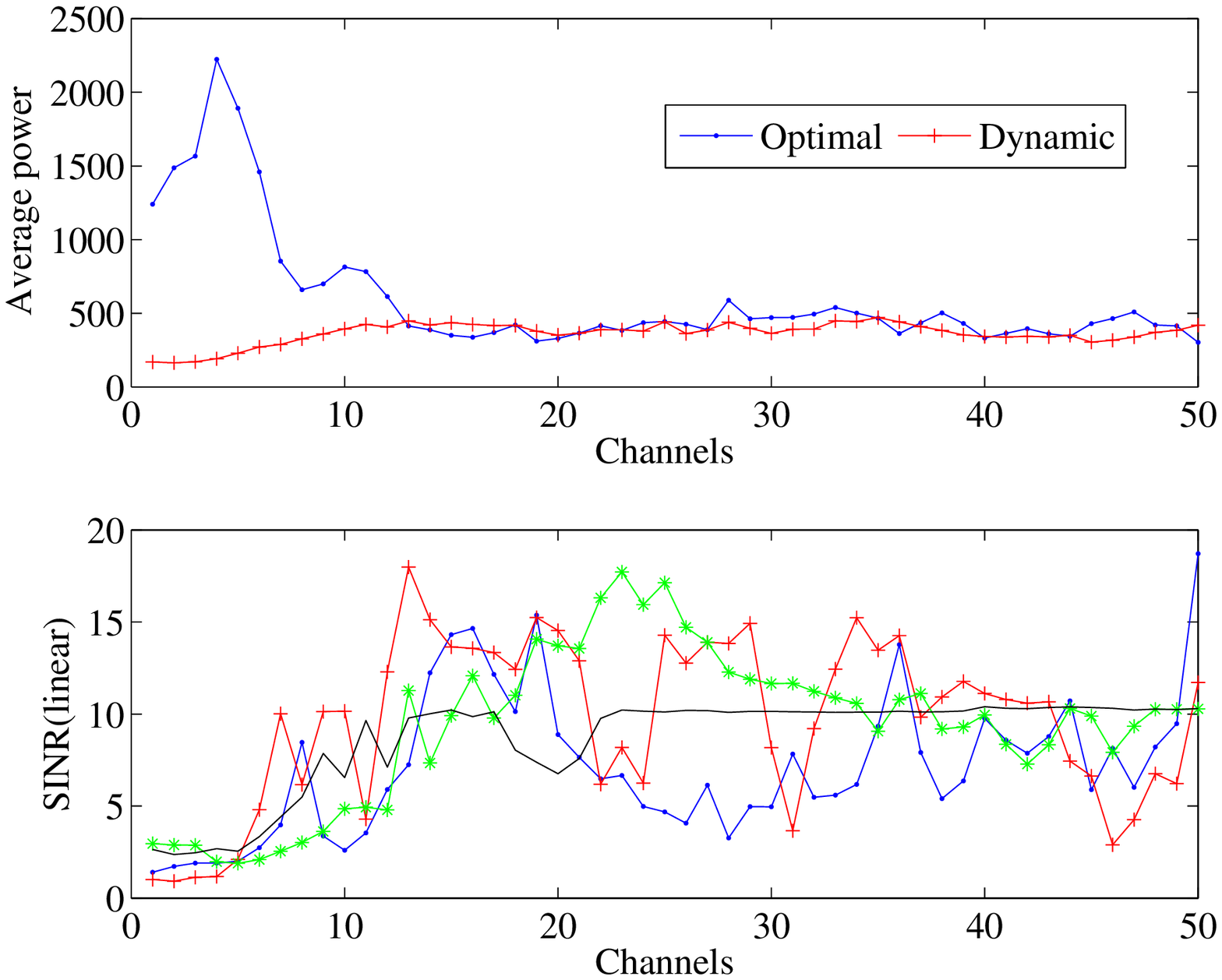}
\caption{Total transmit power and average user SINR for a single track of 50 channels, $\rho = 1$, $\boldsymbol{\gamma} = \mathbf{10}$, $\boldsymbol{\sigma} = \mathbf{\sqrt{10}} $ $\zeta = 0.01$, $N_t = 4$, $N_u = 4$, $N_b =2$.}
\label{fig:singletrack1}
\end{figure}
\begin{figure}
\centering\includegraphics[width=21pc]{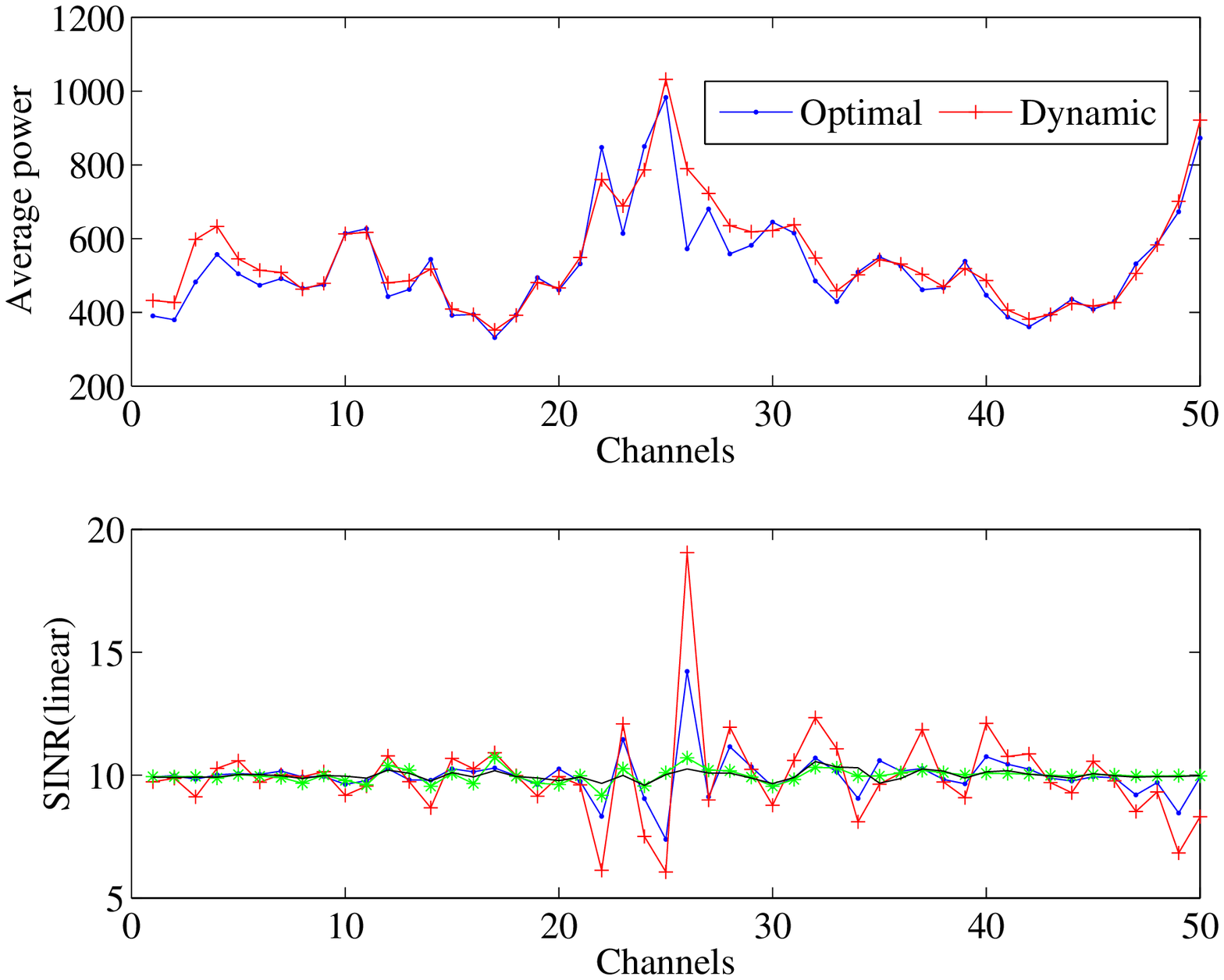}
\caption{Total transmit power and average user SINR for a single track of 50 channels, $\rho = 50$, $\boldsymbol{\gamma} = \mathbf{10}$,$\boldsymbol{\sigma} = \mathbf{\sqrt{10}}$,  $\zeta = 0.01$, $N_t = 4$, $N_u = 4$, $N_b =2$. The plot below illustrates each of the users perceived SINR.}
\label{fig:singletrack50}
\end{figure}
\begin{figure}
\centering\includegraphics[width=21pc]{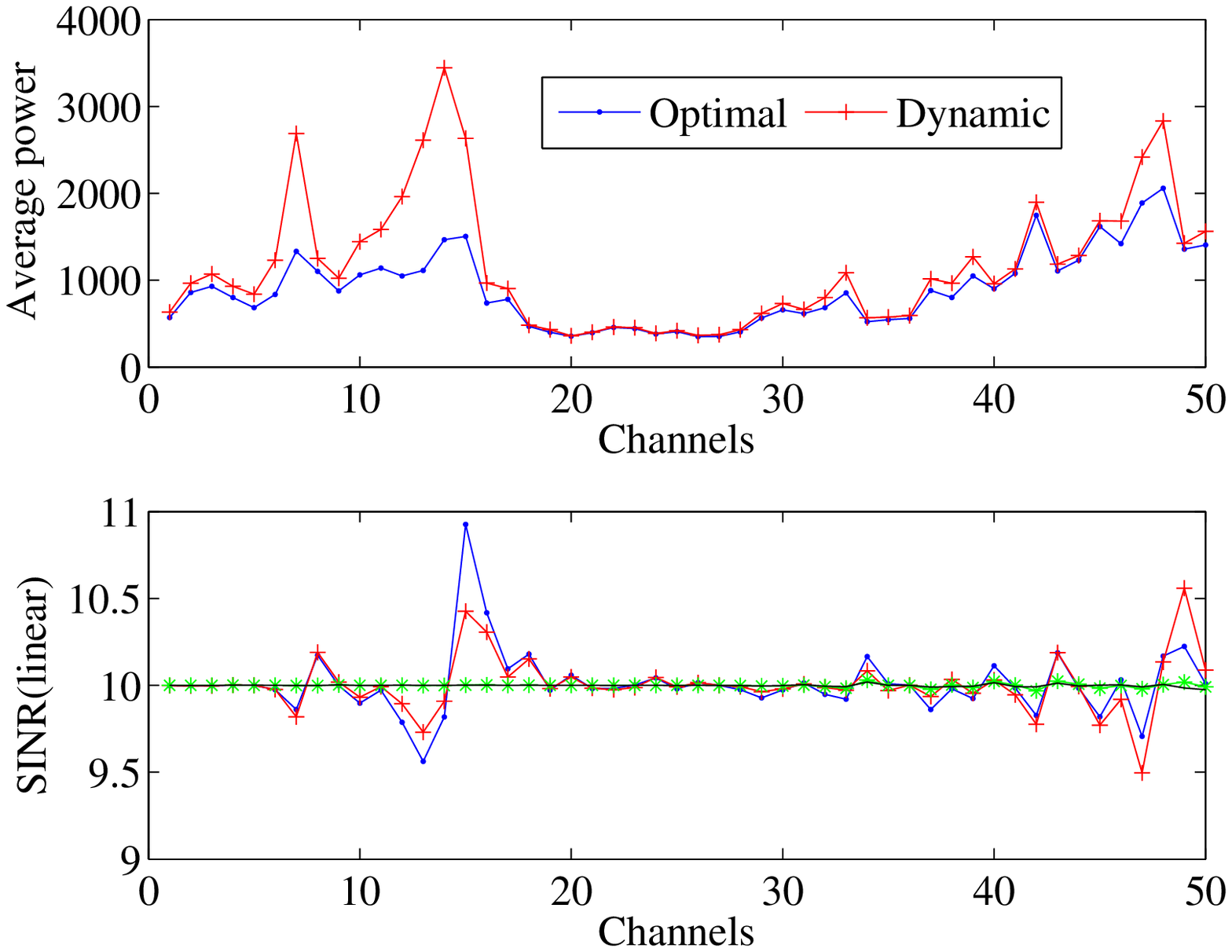}
\caption{Total transmit power and average user SINR for a a single track of 50 channels, $\rho = 1000$, $\boldsymbol{\gamma} = \mathbf{10}$, $\boldsymbol{\sigma} = \mathbf{\sqrt{10}}$ $\zeta = 0.01$, $N_t = 4$, $N_u = 4$, $N_b = 2$. The plot below illustrates each of the users perceived SINR.}
\label{fig:singletrack1000}
\end{figure}
\begin{figure}
\centering\includegraphics[width=21pc]{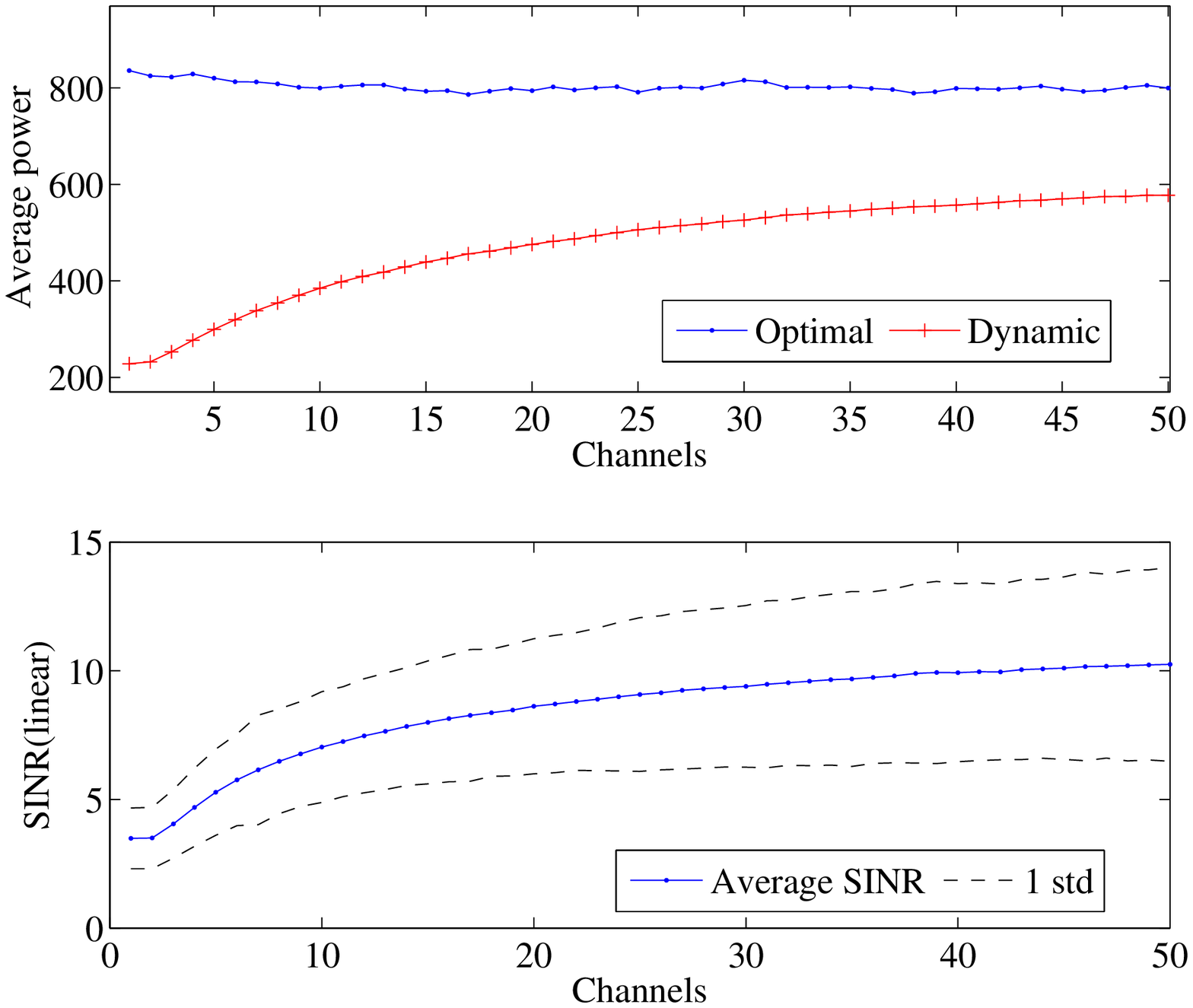}
\caption{Total average transmit power and average user SINR for 10000 tracks of 50 channels, $\rho = 1$, $\boldsymbol{\gamma} = \mathbf{10}$, $\boldsymbol{\sigma} = \mathbf{\sqrt{10}}$ $\zeta = 0.01$, $N_t = 4$, $N_u = 4$, $N_b = 2$.} 
\label{fig:average1}
\end{figure}
\begin{figure}
\centering\includegraphics[width=21pc]{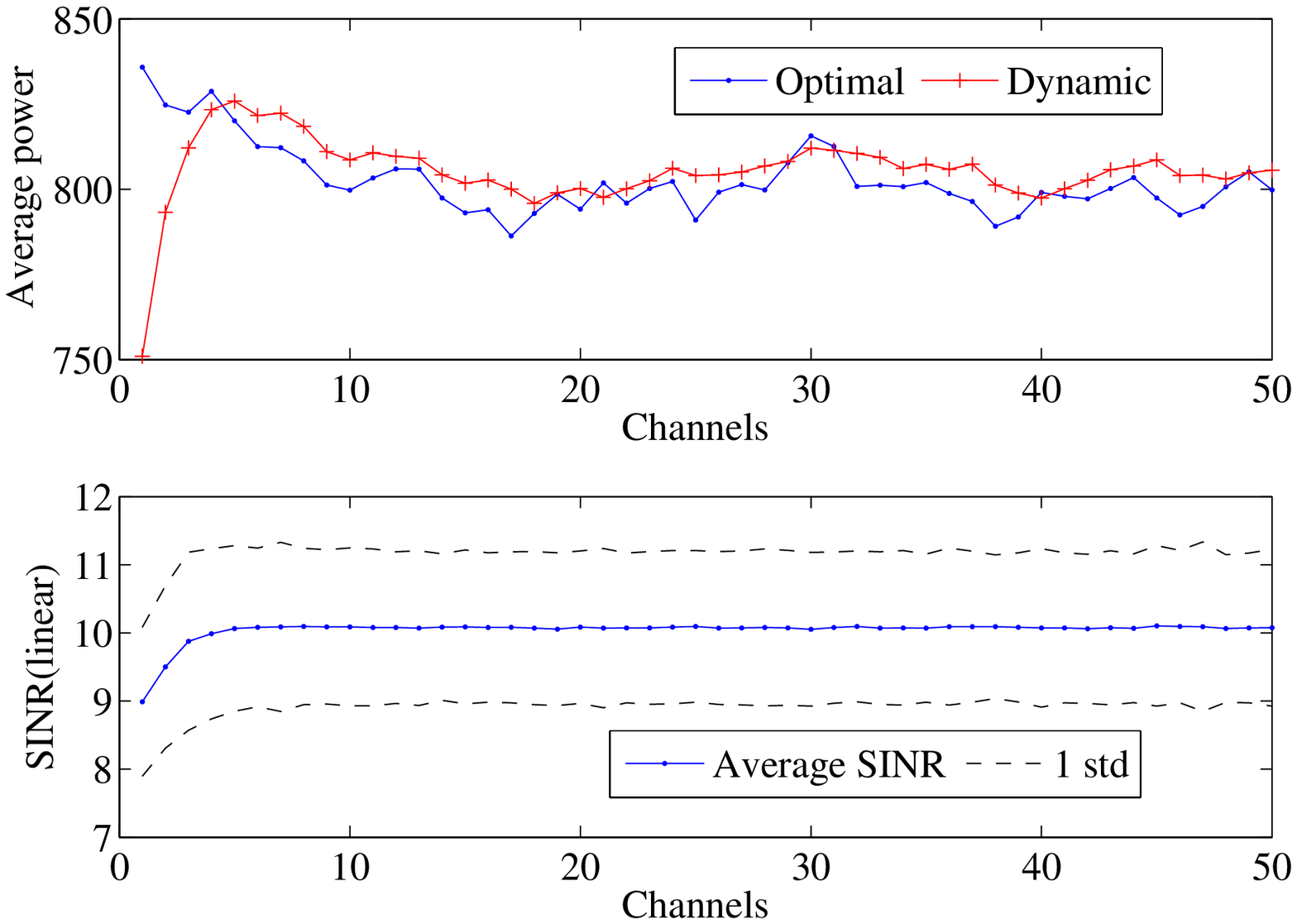}
\caption{Total average transmit power and average user SINR for 10000 tracks of 50 channels, $\rho = 50$, $\boldsymbol{\gamma} = \mathbf{10}$, $\boldsymbol{\sigma} = \mathbf{\sqrt{10}}$, $\zeta = 0.01$, $N_t = 4$, $N_u = 4$, $N_b = 2$.}
\label{fig:average50}
\end{figure}
\begin{figure}
\centering\includegraphics[width=21pc]{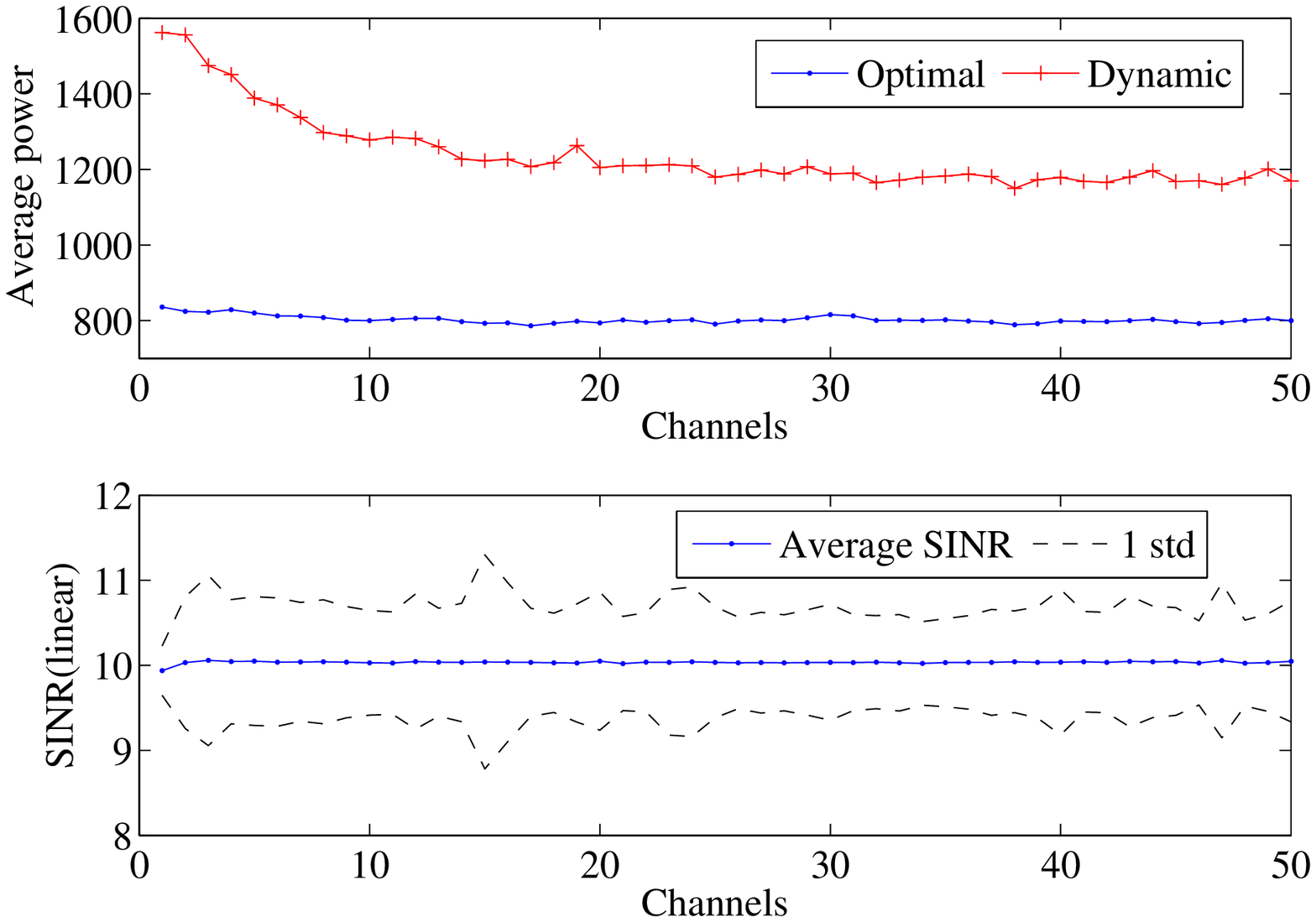}
\caption{Total average transmit power and average user SINR for 10000 tracks of 50 channels, $\rho = 1000$, $\boldsymbol{\gamma}= \mathbf{10}$, $\boldsymbol{\sigma} = \mathbf{\sqrt{10}}$, $\zeta = 0.01$, $N_t = 4$, $N_u = 4$, $N_b = 2$.}
\label{fig:average1000}
\end{figure}

\section{Conclusions \label{section:conclusions}}
This paper shows that ADMM can be used in order to dynamically, and in a distributed manner, follow the set of optimal beam-formers given that the channel varies slowly enough. This is done even though the strong convexity assumption is broken when the problem is written in a ready-to-distribute manner. We have presented a novel approach to show the tracking ability of an algorithm that does not rely on an explicit convergence rate and therefore, allows to us to relax the strong convexity requirement. In particular, the strong convexity requirement is replaced by continuity requirements on the optimal point with respect to the problem's parameters. Additionally, some insights regarding the effect of the step-size on the algorithm's tracking ability are provided.


\appendices
\section{\label{section:appendix1}}
\begin{proof} By writing the KKT conditions of the problem in \eqref{eq:finaldistribute}  it can be shown that it holds that $ \mathbf{E}^T\boldsymbol{\nu}^{\star} = \mathbf{0}$. Additionally, provided that $\boldsymbol{\nu}^{[0]}$ is initialized fulfilling $\mathbf{E}^T\boldsymbol{\nu}^{[0]} = 0$, we can guarantee that at each ADMM iterate, regardless of the channel, will satisfy $\mathbf{E}^T\boldsymbol{\nu}^{[i]}= 0$. This is intuitively sound since $\mathbf{E}^T\boldsymbol{\nu}$ implies that there is a pair of (estimated) dual multipliers that have the same absolute value but opposite sign which will be associated to the copies of the same interference values present at two base stations at a time. This can also be thought of as in \cite{Boyd2007} where the dual variables associated to a consistency constraint are shown to always be 0 when summed over the network. Given this fact, we have that
$||\mathbf{E}(\boldsymbol{\tau}^{[i]\star} - \boldsymbol{\tau}^{[i]}) + (\frac{\boldsymbol{\nu}^{[i]}}{\rho}-\frac{\boldsymbol{\nu}^{[i]\star}}{\rho})||^2 = 
||\mathbf{E}(\boldsymbol{\tau}^{[i]\star}-\boldsymbol{\tau}^{[i]})||^2 + \frac{1}{\rho^2}||\boldsymbol{\nu}^{[i]}-\boldsymbol{\nu}^{[i]\star}||^2$, since the cross product $(\frac{\boldsymbol{\nu}^{[i]}}{\rho} - \frac{\boldsymbol{\nu}^{[i]\star}}{\rho})^T(\mathbf{E}(\boldsymbol{\tau}^{[i]\star} - \boldsymbol{\tau}^{[i]})) = 0$. Note that this expression is nothing but the Lyapunov function scaled by $\frac{1}{\rho}$. Hence we have that $||\mathbf{E}\boldsymbol{\tau}^{[i]\star} - \frac{\boldsymbol{\nu}^{[i]\star}}{\rho} - (\mathbf{E}\boldsymbol{\tau}^{[i]} - \frac{\boldsymbol{\nu}^{[i]}}{\rho})||^2 \leq \frac{\mu_{\text{l}}}{\rho}$. In the sequel, we equivalently 
 rewrite \eqref{eq:objective} as
$\sum_{j\in \mathcal{U}(b)}||\mathbf{w}_{j}||^2 + \frac{\rho}{2}||\mathbf{t}_b - \mathbf{E}_b\boldsymbol{\tau}^{[i-1]} + \frac{\boldsymbol{\nu}_b^{[i-1]}}{\rho}||^2$ in order to use the just derived bound. Note that since the base stations do not share any variable (they share copies) solving each of the problems in \eqref{eq:ADMMdistr} locally at each of the base stations is equivalent to solving a problem where the feasible set is the Cartesian product of feasible sets and the objective function is nothing more than the sum of objective functions. This leads to the following optimization problem:
\begin{subequations}
\label{eq:appendixdistrib}
\begin{align}
&\underset{\mathbf{w},\mathbf{t}}{\text{min}} & ||\mathbf{w}||^2 + \frac{\rho}{2}||\mathbf{t} - \mathbf{E}\boldsymbol{\tau}^{[i-1]} + \frac{\boldsymbol{\nu}^{[i-1]}}{\rho}||^2 \\
&\text{s.t.} & (\mathbf{w},\mathbf{t}) \in \mathcal{WT}(\mathbf{H}^{[i]}),
\end{align}
\end{subequations}
where $\mathcal{WT}(\mathbf{H})$ defined in \eqref{eq:primalset} denotes the feasible set for all beam-formers and interference estimates, i.e. the constraints of all base stations in \eqref{eq:type1} and \eqref{eq:type2}.
Define for the sake of simplicity $\mathbf{y}^{[i-1]} \triangleq \mathbf{E}\boldsymbol{\tau}^{[i-1]} - \frac{\boldsymbol{\nu}^{[i-1]}}{\rho}$. Note that given the optimal set of dual multipliers and consistency variables the problem in \eqref{eq:appendixdistrib} yields the optimal solution to \eqref{eq:finaldistribute} and that the feasible set is only dependent on the channel $\mathbf{H}^{[i]}$ . Define the optimal parameter  $\mathbf{y}^{[i]\star} \triangleq \mathbf{E}\boldsymbol{\tau}^{\star}(\mathbf{H}^{[i]}) - \frac{\boldsymbol{\nu}^{[i]\star}(\mathbf{H}^{[i]})}{\rho}$. Then, the objective function can be equivalently replaced by $||\mathbf{w}||^2 + \frac{\rho}{2}||\mathbf{t}||^2 + \rho \mathbf{y}^T\mathbf{t} = f(\mathbf{w}) + g(\mathbf{t}) + h(\mathbf{t},\mathbf{y}), $ where $f(\mathbf{w}) = ||\mathbf{w}||^2,$ $g(\mathbf{t}) = \frac{\rho}{2}||\mathbf{t}||^2$ and finally $h(\mathbf{t},\mathbf{y}) = \rho \mathbf{y}^T\mathbf{t}$. 
Given the parameters $\mathbf{y}^{[i-1]}$, which is a concatenation of $\mathbf{y}_b^{[i-1]}$ defined earlier as $\mathbf{E}_b\boldsymbol{\tau}^{[i-1]} - \frac{\boldsymbol{\nu}_b^{[i-1]}}{\rho}$ , we have that
\begin{equation}
\begin{aligned}
\label{eq:1stconvex}
\nabla_{\mathbf{w}}f(\mathbf{w}^{[i]})^T(\mathbf{w}^{[i]\star} - \mathbf{w}^{[i]}) + \nabla_{\mathbf{t}}g(\mathbf{t}^{[i]})^T(\mathbf{t}^{[i]\star} - \mathbf{t}^{[i]}) + \\
+ \nabla_{\mathbf{t}}h(\mathbf{t}^{[i]},\mathbf{y}^{[i-1]})^T(\mathbf{t}^{[i]\star} - \mathbf{t}^{[i]}) \geq 0,
\end{aligned}
\end{equation}
where $(\mathbf{w}^{[i]},\mathbf{t}^{[i]})$ is optimal given $\mathbf{y}^{[i-1]}$ and $(\mathbf{w}^{[i]\star},\mathbf{t}^{[i]\star})$ is optimal given $\mathbf{y}^{[i]\star}$. We also have
\begin{equation}
\begin{aligned}
\label{eq:2ndconvex}
\nabla_{\mathbf{w}}f(\mathbf{w}^{[i]\star})^T(\mathbf{w}^{[i]} - \mathbf{w}^{[i]\star}) + 
\nabla_{\mathbf{t}}g(\mathbf{t}^{[i]\star})^T(\mathbf{t}^{[i]} - \mathbf{t}^{[i]\star}) + \\
+ \nabla_{\mathbf{t}}h(\mathbf{t}^{[i]\star},\mathbf{y}^{[i]\star})^T(\mathbf{t}^{[i]} - \mathbf{t}^{[i]\star}) \geq 0.
\end{aligned}
\end{equation}
By adding \eqref{eq:1stconvex} and \eqref{eq:2ndconvex}, and using the strong convexity of $f$ and $g$
\begin{equation}
\begin{aligned}
&(\mathbf{y}^{[i]\star} - \mathbf{y}^{[i-1]})^T(\mathbf{t}^{[i]\star} - \mathbf{t}^{[i]}) \geq  \\
& ||\mathbf{w}^{[i]} - \mathbf{w}^{[i]\star}||^2 + ||\mathbf{t}^{[i]} - \mathbf{t}^{[i]\star}||^2.
\end{aligned}
\end{equation}
In turn, the first term can be upper bounded by
\begin{equation}
\label{eq:processbound}
\begin{aligned}
||\mathbf{t}^{[i]\star} - \mathbf{t}^{[i]}|| ||\mathbf{y}^{[i]\star} - \mathbf{y}^{[i-1]}|| \leq \sqrt{c}(||\mathbf{t}^{[i]\star} - \mathbf{E}\boldsymbol{\tau}^{[i]\star}|| \\ + ||\mathbf{E}(\boldsymbol{\tau}^{[i]\star}-\boldsymbol{\tau}^{[i]\star})|| + ||\mathbf{t}^{[i]} - \mathbf{E}\boldsymbol{\tau}^{[i]}||),
\end{aligned}
\end{equation}
where $c$ represents the bound on the Lyapunov function as in \eqref{eq:finalboundlyapunov}.  Note that the first term in the RHS of \eqref{eq:processbound} is 0 since we are dealing with optimal points. Additionally, the second term can be again bounded by $\sqrt{c}$. The third term is a scaled version of the primal residual and can be also bounded by the Lyapunov function, since one can not perform a decrease larger than its current value. Hence, we conclude that
\begin{equation}
||\mathbf{w}^{[i]} - \mathbf{w}^{[i]\star}||^2 + ||\mathbf{t}^{[i]} - \mathbf{t}^{[i]\star}||^2 \leq \left(1 + \frac{1}{\sqrt{\rho}}\right)c,
\end{equation}
for $i \rightarrow \infty$ and hence, given that $\mathbf{w}^{[i]}$ is a vectorized version of $\mathbf{W}^{[i]}$ we have
\begin{equation}
\begin{aligned}
&\underset{i \rightarrow \infty}{\text{lim sup}} &||\mathbf{W}^{[i]} - \mathbf{W}^{[i]\star}||^2_{\mathrm{F}} \leq \left(1 + \frac{1}{\rho}\right)c
\end{aligned}
\end{equation}
\end{proof}
\section{\label{section:appendix2}}
\begin{proof} After iteration $i$ using the corresponding channels $\mathbf{H}^{[i]}$, each base station has found a set of beam-formers and local copies of interference values $t_{mk}^{(b)}$ that fulfill the SINR constraints tightly. However, since before convergence ADMM does not guarantee primal feasibility, the local interference estimate might not match the perceived interference when the obtained beam-formers $\mathbf{W}^{[i]}$ are used, i.e. different base-stations may disagree on how much they are interfering each other and hence the interfering base station will cause more interference than predicted by the base station whose user is suffering the interference. We therefore aim to find the worst case perceived SINR. The proof will be performed for user $k$ associated to base station $b$. In particular we have that base station $b$ has performed an ADMM step yielding beam-formers and interference values such that
\begin{equation}
\label{eq:idealsinr}
\frac{|\mathbf{h}_{b(k)k}^H\mathbf{w}_k|^2}{\sum_{i \in \mathcal{U}(b(k)) \setminus k} |\mathbf{h}_{b(k)k}^H\mathbf{w}_i|^2 + \sum_{m \neq b}t_{mk}^{(b)2} + \sigma_k^2} = \gamma_k.
\end{equation}
However, the perceived SINR satisfies
\begin{equation}
\label{eq:perceivedsinr}
\frac{|\mathbf{h}_{b(k)k}^H\mathbf{w}_k|^2}{\sum_{i \in \mathcal{U}(b(k)) \setminus k} |\mathbf{h}_{b(k)k}^H\mathbf{w}_i|^2 + \sum_{m \neq b}t_{mk}^{(m)2} + \sigma_k^2} \geq \gamma_k - \epsilon_{2k},
\end{equation}
where $\epsilon_{2k}$ is the loss of SINR at user $k$ and is the quantity we wish to upper bound. For notational simplicity we define $\mathbf{t}_{bk}'$ and $\mathbf{t}_{mk}'$ as the vectors containing in each of their components the interference estimates of base station $b$ appearing in \eqref{eq:idealsinr} and analogously for $\mathbf{t}_{mk}'$ with \eqref{eq:perceivedsinr}, implying that $||\mathbf{t}_{bk}'||^2 = \sum_{m \neq b}t_{mk}^{(b)2}$ and $||\mathbf{t}_{mk}'||^2 = \sum_{m \neq b} t_{mk}^{(m)2}$.

By writing the difference between \eqref{eq:idealsinr} and \eqref{eq:perceivedsinr} and simplifying we obtain
\begin{subequations}
\begin{align}
\frac{\gamma_k (||\mathbf{t}_{mk}'||^2 - ||\mathbf{t}_{bk}'||^2)}{\sum_{i \in \mathcal{U}(b(k)) \setminus k} |\mathbf{h}_{b(k)k}^H\mathbf{w}_i|^2 + ||\mathbf{t}_{mk}'||^2 + \sigma_k^2} \leq \\
\frac{\gamma_k (||\mathbf{t}_{mk}'||^2 - ||\mathbf{t}_{bk}'||^2)}{ ||\mathbf{t}_{mk}'||^2 + \sigma_k^2}
\end{align}
\end{subequations}
where we have used that $||\mathbf{t}_{mk}'|| \geq ||\mathbf{t}_{bk}'||$ since we are interested in bounding the worst case scenario. In particular, the worst perceived interference, by a specific user, will occur when the user is expected to not be interfered at all, i.e. zero forced by other base stations, but it is however interfered. Recall now, that the primal residual, i.e. $||\mathbf{t} - \mathbf{E}\boldsymbol{\tau}||^2$ acts as a lower bound in the Lyapunov's function decrease \eqref{eq:decreaselowerbound}. Hence, the primal residual can not attain a value larger than the Lyapunov function itself. Additionally it has been shown in Section \ref{section:trackingwithadmm} that \eqref{eq:limitlyapunov} holds. We have also seen that when the channel changes the Lyapunov function with no update can be upper bounded as follows:
\begin{equation}
V(\boldsymbol{\nu}^{[i]},\boldsymbol{\tau}^{[i]},\mathbf{H}^{[i+1]}) \leq c, 
\end{equation}
where $c$ is defined in \eqref{eq:finalboundlyapunov}. Hence, the primal residual can not attain values larger than $c$. Consequently the term $||\mathbf{t}_{mk}' - \mathbf{t}_{bk}'||^2 \leq \frac{4 c}{\rho}$. We then have that
\begin{subequations}
\begin{align}
\label{eq:beforesplit}
\frac{\gamma_k (||\mathbf{t}_{mk}'||^2 - ||\mathbf{t}_{bk}'||^2))}{\ ||\mathbf{t}_{mk}'||^2 + \sigma_k^2} \leq \\
\frac{\gamma_k||\mathbf{t}_{mk}'||^2}{||\mathbf{t}_{mk'}||^2 + \sigma_k^2} \leq \frac{\gamma_k 4c}{\rho\sigma_k^2 + 4c}.
\end{align}
\end{subequations}
Thus yielding the upper bound
\begin{equation}
 \epsilon_{2k} \leq \frac{\gamma_k 4 \epsilon_1}{\rho \sigma_k^2 + 4\epsilon_1}.
\end{equation}
\end{proof}

\section*{Acknowledgment}
The authors would like to thank Prof. W. Yu for helpful discussions regarding the proof of Theorem 1 in \cite{Dahrouj2010}.
\bibliographystyle{IEEEbib.bst}
\bibliography{MyCollection}
\ifCLASSOPTIONcaptionsoff
  \newpage
\fi

\end{document}